\documentclass[11pt]{article}
\usepackage[margin=1in]{geometry} 
\usepackage{amssymb,amsfonts,amsmath,bbm,mathrsfs,stmaryrd}
\usepackage{xcolor}
\usepackage{url}
\usepackage{ytableau}

\usepackage[colorlinks,
             linkcolor=black!75!red,
             citecolor=blue,
             pdftitle={Mackey},
             pdfauthor={Alexandru Chirvasitu},
             pdfproducer={pdfLaTeX},
             pdfpagemode=None,
             bookmarksopen=true
             bookmarksnumbered=true]{hyperref}

\usepackage{tikz,tikz-cd}
\usetikzlibrary{matrix,arrows,calc,decorations.pathreplacing,decorations.markings,intersections,shapes.geometric,through,fit,shapes.symbols,positioning,decorations.pathmorphing,snakes}

\usepackage[amsmath,thmmarks,hyperref]{ntheorem}
\usepackage{cleveref}

\crefname{section}{Section}{Sections}
\crefformat{section}{#2Section~#1#3} 
\Crefformat{section}{#2Section~#1#3} 

\crefname{subsection}{Subsection}{Subsections}
\crefformat{subsection}{#2Subsection~#1#3} 
\Crefformat{subsection}{#2Subsection~#1#3} 

\theoremstyle{plain}

\newtheorem{lemma}{Lemma}[section]
\newtheorem{proposition}[lemma]{Proposition}
\newtheorem{corollary}[lemma]{Corollary}
\newtheorem{theorem}[lemma]{Theorem}

\theoremstyle{nonumberplain}

\newtheorem{thbis}{\Cref{th.soc_filt} bis}
\newtheorem{corbis}{\Cref{cor.self-dual} bis}

\theoremstyle{plain}
\theorembodyfont{\upshape}
\theoremsymbol{\ensuremath{\blacklozenge}}

\newtheorem{definition}[lemma]{Definition}
\newtheorem{notation}[lemma]{Notation}
\newtheorem{example}[lemma]{Example}
\newtheorem{remark}[lemma]{Remark}

\crefname{definition}{definition}{definitions}
\crefformat{definition}{#2definition~#1#3} 
\Crefformat{definition}{#2Definition~#1#3} 

\crefname{ex}{example}{examples}
\crefformat{example}{#2example~#1#3} 
\Crefformat{example}{#2Example~#1#3} 

\crefname{remark}{remark}{remarks}
\crefformat{remark}{#2remark~#1#3} 
\Crefformat{remark}{#2Remark~#1#3} 

\crefname{convention}{convention}{conventions}
\crefformat{convention}{#2convention~#1#3} 
\Crefformat{convention}{#2Convention~#1#3}

\crefname{lemma}{lemma}{lemmas}
\crefformat{lemma}{#2lemma~#1#3} 
\Crefformat{lemma}{#2Lemma~#1#3} 

\crefname{proposition}{proposition}{propositions}
\crefformat{proposition}{#2proposition~#1#3} 
\Crefformat{proposition}{#2Proposition~#1#3} 

\crefname{corollary}{corollary}{corollaries}
\crefformat{corollary}{#2corollary~#1#3} 
\Crefformat{corollary}{#2Corollary~#1#3} 

\crefname{theorem}{theorem}{theorems}
\crefformat{theorem}{#2theorem~#1#3} 
\Crefformat{theorem}{#2Theorem~#1#3} 

\crefname{assumption}{assumption}{Assumptions}
\crefformat{assumption}{#2assumption~#1#3} 
\Crefformat{assumption}{#2Assumption~#1#3} 

\crefname{equation}{}{}
\crefformat{equation}{(#2#1#3)} 
\Crefformat{equation}{(#2#1#3)}

\theoremstyle{nonumberplain}
\theoremsymbol{\ensuremath{\blacksquare}}

\newtheorem{proof}{Proof}

\newtheorem{proof_of_3d_derived_ff}{Proof of \Cref{th.3d_derived_ff}}
\newtheorem{proof_of_univ}{Proof of \Cref{th.univ}}

\newtheorem{proof_of_simples}{Proof of \Cref{pr.simples}}
\newtheorem{proof_of_Kosz}{Proof of \Cref{th.Kosz}}
\newtheorem{proof_of_3d_univ}{Proof of \Cref{th.3d_univ}}
\newtheorem{proof_of_inj_env}{Proof of Proposition~\ref{pr.inj_env}}

\newcommand\bC{{\mathbb C}}

\newcommand\bK{{\mathbb K}}
\newcommand\bN{{\mathbb N}}

\newcommand\bT{{\mathbb T}}
\newcommand\bZ{{\mathbb Z}}

\newcommand\cA{{\mathcal A}}
\newcommand\cC{{\mathcal C}}
\newcommand\cD{{\mathcal D}}

\newcommand\cM{{\mathcal M}}
\newcommand\cS{{\mathcal S}}

\newcommand\fg{\mathfrak{g}}
\newcommand\fG{\mathfrak{G}}

\newcommand\fgl{\mathfrak{gl}}
\newcommand\fsl{\mathfrak{sl}}
\newcommand\fo{\mathfrak{o}}
\newcommand\fsp{\mathfrak{sp}}

\newcommand\ol{\overline}

\DeclareMathOperator{\id}{id}
\DeclareMathOperator{\End}{\mathrm{End}}
\DeclareMathOperator{\Hom}{\mathrm{Hom}}
\DeclareMathOperator{\Ext}{\mathrm{Ext}}

\DeclareMathOperator{\soc}{\mathrm{soc}}
\DeclareMathOperator{\usoc}{\underline{\mathrm{soc}}}

\DeclareMathOperator{\fin}{\cat{fin}}

\DeclareMathOperator{\sym}{\cat{Sym}}
\newcommand{\define}[1]{{\em #1}}

\newcommand{\cat}[1]{\textsc{#1}}

\newcommand{\qedhere}{\mbox{}\hfill\ensuremath{\blacksquare}}

\title{Ordered tensor categories and representations of the Mackey Lie algebra of infinite matrices}
\author{Alexandru Chirvasitu,\quad   Ivan Penkov}

\begin{document}

\date{}

\newcommand{\Addresses}{{% additional braces for segregating \footnotesize
  \bigskip
  \footnotesize

  \textsc{Department of Mathematics, University
    of Washington, Seattle, WA 98195-4350, USA}\par\nopagebreak
  \textit{E-mail address}: \texttt{chirva@uw.edu}

  \medskip

  \textsc{Jacobs University Bremen, Campus Ring 1, 28759 Bremen, Germany}\par\nopagebreak
  \textit{E-mail address}: \texttt{i.penkov@jacobs-university.de}
}}

\maketitle

\begin{center}
  {\it To Jean-Louis Koszul on his 95th birthday}
\end{center}

\begin{abstract}
We introduce (partially) ordered Grothendieck categories and apply results on their structure to the study of categories of representations of the Mackey Lie algebra of infinite matrices $\fgl^M\left(V,V_*\right)$.  Here $\fgl^M\left(V,V_*\right)$ is the Lie algebra of endomorphisms of a nondegenerate pairing of countably infinite-dimensional vector spaces $V_*\otimes V\to\bK$, where $\bK$ is the base field.  Tensor representations of $\fgl^M\left(V,V_*\right)$ are defined as arbitrary subquotients of finite
direct sums of tensor products $(V^*)^{\otimes m}\otimes
(V_*)^{\otimes n}\otimes V^{\otimes p}$ where $V^*$ denotes the algebraic dual of $V$. The category $\bT^3_{\fgl^M\left(V,V_*\right)}$
which they comprise, extends a category  $\bT_{\fgl^M\left(V,V_*\right)}$ previously studied in
\cite{DPS,PS2,SS}, and our main result is that $\bT^3_{\fgl^M\left(V,V_*\right)}$ is a
finite-length, Koszul self-dual, tensor category with a certain
universal property that makes it into a ``categorified algebra''
defined by means of a handful of generators and relations.  
%Motivated by this result we develop a theory of ordered tensor
%categories which allows us to place the category $\bT^3_{\fgl^M\left(V,V_*\right)}$ into
%a reasonably general context. 
This result uses essentially the general properties of ordered Grothendieck categories, which yield also simpler proofs of some facts about the category $\bT_{\fgl^M\left(V,V_*\right)}$ established in~\cite{PS2}.   Finally, we discuss the extension
of $\bT^3_{\fgl^M\left(V,V_*\right)}$ by the algebraic dual $(V_*)^*$ of $V_*$.

\end{abstract}


\vspace{.5cm}


\tableofcontents

%%%%%%%%%%%%%%%%%%%%%%%%%%%%%%%%%%%%%%%%%%%%%%%%%%%%%%%%%%%%%%%%%%%%%%%%%%%%%%%%%%%%%%%%%%%%%%%%%%%%%%%%%%%%%%%%%%
%%%%%%%%%%%%%%%%%%%%%%%%%%%%%%%%%%%%%%%%%%%%%%%%%%%%%%%%%%%%%%%%%%%%%%%%%%%%%%%%%%%%%%%%%%%%%%%%%%%%%%%%%%%%%%%%%%
\section*{Introduction}\label{se.intro}

Despite the fact that matrix Lie algebras are part of the core of classical Lie theory, Lie algebras of infinite matrices present many challenges concerning their structure and representations.  Our study is motivated by the
%A motivation for this study is our 
desire to understand certain natural categories of tensor representations of an infinite-dimensional matrix Lie algebra. These categories are analogues of the category of finite-dimensional representations of the Lie algebra $\fsl(n)$ (or more generally, of any simple finite-dimensional Lie algebra) but are not semisimple. 

A first example of such a category is $\bT_{\fg}$, the category of tensor modules over a finitary simple Lie algebra $\fg\cong \fsl(\infty)$, $\fo(\infty)$ or $\fsp(\infty)$. This category was introduced independently in \cite{DPS} and \cite{SS} and has been further studied in \cite{PS2} (see also \cite{PS}). Some of the essential attributes of the category $\bT_{\fg}$ are that its objects have finite length and that it is a nonrigid Koszul tensor category.  

A more elaborate example is the category $\widetilde{\mathrm{Tens}}_{\fg}$ for $\fg$ as above, introduced in \cite{PS1}. It contains both defining representations of $\fg$ as well as their algebraic duals, but is no longer a finite-length category. Notably the objects of $\widetilde{\mathrm{Tens}}_{\fg}$ have finite Loewy length. 

The paper \cite{PS2} introduces a class of infinite-dimensional, non-locally finite matrix Lie algebras called \define{Mackey Lie algebras}. A first example is the algebra $\fgl^M\left(V,V_*\right)$ of all endomorphisms of a countable-dimensional vector space $V$ which preserve a fixed nondegenerate pairing $V\otimes V_*\to \bK$ ($\bK$ is the base field, $\mathrm{char}\,\bK=0$), $V_*$ being also a countable-dimensional vector space.
%whose duals preserve a countable-dimensional subspace $V_*\subset V^*$.  
In matrix terms, $\fgl^M\left(V,V_*\right)$ consists of infinite matrices each of whose rows and columns has at most finitely many nonzero entries.  This Lie algebra has two obvious modules: $V$ and $V_*$.  In addition, interestingly, the algebraic dual $V^*$ of $V$, although not simple over $\fgl^M\left(V,V_*\right)$, has finite length: the socle of $V^*$ as a $\fgl^M(V,V_*)$-module equals $V_*$, and the quotient $V^*/V_*$ is a simple $\fgl^M(V,V_*)$-module. Moreover, as shown by the first author in \cite{Chi14}, the tensor category $\bT^3_{\fgl^M\left(V,V_*\right)}$ generated by $V$, $V_*$ and $V^*$ is a finite-length category. 

With the categories $\bT_{\fg}$ and $\bT^3_{\fgl^M\left(V,V_*\right)}$ in mind, we introduce in this paper a general notion of ordered Grothendieck category. These are Grothendieck categories for which the indecomposable injectives are parametrized by a partially ordered set with finite downward sequences; the precise definition is given in \Cref{se.inj} below. Our main result regarding ordered Grothendieck categories is that they are equivalent to categories of comodules over semiperfect coalgebras, and that they are universal in a certain sense to be made precise below.  % In addition, we give a quick description of the indecomposable injectives in an ordered Grothendieck category.

Applied to the Grothendieck closure of $\bT_{\fg}$ for $\fg=\fsl(\infty)$, $\fo(\infty)$ or $\fsp(\infty)$, this result simplifies proofs of some main results in \cite{DPS}. It is also crucial in our study of the category $\bT^3_{\fgl^M\left(V,V_*\right)}$: we give an explicit parametrization of the simple objects of $\bT^3_{\fgl^M\left(V,V_*\right)}$ via triples of Young diagrams, characterize the indecomposable injectives explicitly, find the blocks of $\bT_{\fgl^M\left(V,V_*\right)}^3$, and, last but not least, prove the Koszulity of $\bT^3_{\fgl^M\left(V,V_*\right)}$. Moreover, we show that $\bT^3_{\fgl^M\left(V,V_*\right)}$ is Koszul self-dual.   We also apply our general universality result to show that the category $\bT^3_{\fgl^M\left(V,V_*\right)}$ is universal in a much stronger sense than a general ordered Grothendieck category.% the universality of the Grothendieck closure $\overline{\bT}^3_{\fgl^M\left(V,V_*\right)}$.

Finally, we take a few first steps in studying the abelian tensor category $\bT^4_{\fgl^M}$ generated by $V$, $V_*$, $V^*$ and $(V_*)^*$. We prove that $\bT^4_{\fgl^M\left(V,V_*\right)}$ is a finite-length category and its simple objects are parametrized by quadruples of Young diagrams. However, the Grothendieck closure of $\bT^4_{\fgl\left(V,V_*\right)}$ is not an ordered Grothendieck category, and therefore we leave a more detailed study of $\bT^4_{\fgl\left(V,V_*\right)}$
% since $\bT^4_{\fgl^M\left(V,V_*\right)}$ lacks some of the finiteness properties we impose on ordered tensor categories in this paper, we leave a more detailed study of this category 
(e.g. injectives in its Grothendieck closure) to the future. 

The formalism of ordered tensor categories, as introduced in this paper, will very likely be applicable to categories of tensor representations of diagonal infinite-dimensional Lie algebras such as $\fgl(2^\infty)$. 

In addition to the intrinsic interest of the formalism of ordered tensor categories, we believe the investigations conducted in this paper fit into an algebraic framework for studying topics that are typically the domain of functional analysis. Indeed, $V_*$ can be regarded as a topological dual to $V$ for a certain topology on the latter (the Mackey topology determined by the pairing $V\otimes V_*\to \bK$).  In this sense, we are examining the interaction between duals associated to different topologies (i.e. $V_*$ and $V^*$); this is one of the main themes in the study of locally convex topological vector spaces.

%%%%%%%%%%%%%%%%%%%%%%%%%%%%%%%%%%%%%%%%%%%%%%%%%%%%%%%%%%%%%%%%%%%%%%%%%%%%%%%%%%%%%%%%%%%%%%%%%%%%%%%%%%%%%%%%%%
\subsection*{Acknowledgements}

We thank Vera Serganova for sharp comments on the topic of this paper. I. P. acknowledges continued partial support by the DFG through the  Priority Program ``Representation Theory'' and through grant PE 980/6-1.

%%%%%%%%%%%%%%%%%%%%%%%%%%%%%%%%%%%%%%%%%%%%%%%%%%%%%%%%%%%%%%%%%%%%%%%%%%%%%%%%%%%%%%%%%%%%%%%%%%%%%%%%%%%%%%%%%%
%%%%%%%%%%%%%%%%%%%%%%%%%%%%%%%%%%%%%%%%%%%%%%%%%%%%%%%%%%%%%%%%%%%%%%%%%%%%%%%%%%%%%%%%%%%%%%%%%%%%%%%%%%%%%%%%%%
\section{Background}\label{se.prel}

All algebras, Lie algebras, coalgebras, etc. are over a field $\bK$ fixed throughout.  Except in \Cref{se.inj}, the characteristic of $\bK$ is assumed to equal $0$.  The superscript $^*$ indicates dual space, i.e., $V^*=\Hom_\bK(V,\bK)$ for a vector space $V$.  The sign $\otimes$ means $\otimes_\bK$, except in \Cref{se.inj} where it denotes the tensor product in an abstract tensor category.

Let $\mathrm{p}:V\otimes V_*\to\bK$ be a fixed nondegenerate pairing ($\bK$-bilinear form) of countably infinite-dimensional vector spaces $V,V_*$ over $\bK$.   G. Mackey studied in his dissertation~\cite{X} arbitrary nondegenerate pairings of infinite-dimensional spaces, and proved that if both $V$ and $V_*$ are countable dimensional, such a pairing is unique up to isomorphism.  In particular, there is a basis $\left\{v_\alpha\right\}$ of $V$ such that $V_*=\mathrm{span}\left\{v_\alpha^*\right\}$, where $v_\alpha^*$ is the dual system of vectors defined by $\mathrm{p}$: $\mathrm{p}\left(v_\alpha\otimes v_\beta^*\right)=\delta_{\alpha\beta}$. % We call such a pairing \emph{diagonalizable}; arbitrary, not necessarily diagonalizable, pairings of infinite-dimenisonal vector spaces have been studied by G. Mackey in his dissertation \cite{X}.

As pointed out in the introduction, the Mackey Lie algebra $\fgl^M\left(V,V_*\right)$ associated with the pairing $\mathrm{p}$ is the Lie algebra of endomorphisms of $\mathrm{p}$, i.e., the Lie subalgebra of $\End_\bK(V)$ 
$$\fgl^M\left(V,V_*\right)=\left\{\varphi\in \End_\bK(V)\,\big|\,\varphi^*\left(V_*\right)\subseteq V_*\right\}\,,$$
where $\varphi^*:V^*\to V^*$ is the endomorphism dual to $\varphi$.  In a basis $\left\{v_\alpha\right\}$ as above, $\fgl^M\left(V,V_*\right)$ consists of infinite matrices each of whose rows and columns has at most finitely many nonzero entries.  More general Mackey Lie algebras have been introduced in {\cite{PS2}}.

%In {\cite{PS2}} also the category $\bT_{\fgl^M\left(V,V_*\right)}$ was introduced.  By definition
Next, $\bT_{\fgl^M\left(V,V_*\right)}$ is the full tensor subcategory of $\fgl^M\left(V,V_*\right)\text{-}\mathrm{Mod}$ whose objects are arbitrary subquotients of finite direct sums of $\fgl^M$-modules of the form $ \left( V_*\right)^{\otimes m}\otimes V^{\otimes n}$.  A main result of {\cite{PS2}} claims that the category $\bT_{\fgl^M\left(V,V_*\right)}$ is naturally equivalent to the similarly defined category $\bT_{\fsl\left(V,V_*\right)}$, where $\fsl\left(V,V_*\right):=\ker \mathrm{p}$ as a  Lie algebra.  %Moreover, the categories $\bT_{\fsl\left(V,V_*\right)}$ are equivalent for all diagonalizable pairings $p:V\otimes V_*\to \bK$ {\cite{PS2}}, hence the same applies to the categories $\bT_{\fgl^M\left(V,V_*\right)}$ for all diagonalizable pairings $p$.

The category $\bT_{\fsl\left(V,V_*\right)}$ has been introduced and studied in \cite{DPS} (and independently also in \cite{SS}).  The main features of $\bT_{\fsl\left(V,V_*\right)}$ are that it is a finite-length Koszul self-dual category.  More precisely, it is shown in \cite{DPS} that $\bT_{\fsl\left(V,V_*\right)}$ is antiequivalent to a category of finite-dimensional modules over an associative algebra $A_{\fsl\left(V,V_*\right)}$ which is an inductive limit of finite-dimensional Koszul self-dual algebras.  Indecomposable injectives in $\bT_{\fsl\left(V,V_*\right)}$, i.e., injective hulls of simple modules, turn out to be precisely arbitrary indecomposabile direct summands of the modules $\left( V_*\right)^{\otimes m}\otimes V^{\otimes n}$.% and each such indecomposable injective is the injective hull of a simple module in $\bT_{\fsl\left(V,V_*\right)}$.  The simple modules in $\bT_{\fsl\left(V,V_*\right)}$ are parametrized by pairs of Young diagrams.

Consider now the Lie algebra $\End_\bK(V)$ for an arbitrary vector space $V$.  This Lie algebra also has a natural tensor category of representations $\bT_{\End_\bK(V)}$ consisting of all subquotients of finite direct sums of the form $\left(V^*\right)^{\otimes m}\otimes V^{\otimes n} $, and as proved in {\cite{PS2}}, the tensor category $\bT_{\End_\bK(V)}$ is naturally equivalent to $\bT_{\fsl\left(V,V_*\right)}$.  The equivalence functor $\bT_{\End_\bK(V)}\rightsquigarrow \bT_{\fsl\left(V,V_*\right)}$ maps $V$ to $V$ and $V^*$ to $V_*$.  However, since $\fgl^M\left(V,V_*\right)$ is a Lie subalgebra of $\End_\bK(V)$, it is interesting to ask whether the tensor category ``generated'' by the restrictions of modules in $\bT_{\End_\bK(V)}$ to $\fgl^M\left(V,V_*\right)$ has good properties.

Studying the latter tensor category is a main topic in the present paper.  The precise definition of the category is as follows: $\bT^3_{\fgl^M\left(V,V_*\right)}$ is the full tensor subcategory of $\fgl^M\left(V,V_*\right)\text{-}\mathrm{Mod}$ whose objects are $\fgl^M\left(V,V_*\right)$-subquotients of finite direct sums of the form $\left(V^*\right)^{\otimes m}\otimes V^{\otimes n}$.  Some first results about the category $\bT^3_{\fgl^M\left(V,V_*\right)}$ have been established in {\cite{Chi14}} by the first author; in particular, it is shown in {\cite{Chi14}} that $\bT^3_{\fgl^M\left(V,V_*\right)}$ is a finite-length tensor category.

In what follows, unless the contrary is stated explicitly, we consider $V$, $V_*$, $V^*$, and $\fgl^M\left(V,V_*\right)$ fixed; we write $\fgl^M$ instead of $\fgl^M\left(V,V_*\right)$ for brevity.

In order to state some further results of {\cite{Chi14}}, let us recall that, given any vector space $Z$ and a Young diagram $\lambda$ (or, equivalently, a partition $\lambda=\left(\lambda_1\geq \lambda_2\geq \ldots\geq \lambda_k>0\right)$), any filling of $\lambda$ turning $\lambda$ into a Young tableau $\bar{\lambda}$ defines a subspace $Z_\lambda \subset Z^{\otimes|\lambda|}$, for $|\lambda|=\sum_{j=1}^k\,\lambda_j$.  %, as the image 
This subspace is the image of the Schur projection $Z^{\otimes|\lambda|}\to Z^{\otimes|\lambda|}$ determined by $\bar{\lambda}$.  If $Z$ is a representation of a Lie algebra, then any $Z_\lambda$ is also a representation of the same Lie algebra, and $Z_\lambda$ depends up to isomorphism only on $\lambda$ and not on $\bar{\lambda}$.

Given two Young diagrams $\mu$, $\nu$, it is proved in \cite{DPS} that the $\fgl^M$-module $\left(V_*\right)_\mu\otimes V_\nu$ is indecomposable (and injective in $\bT_{\fgl^M)}$) and has simple socle.  We denote this socle by $V_{\mu,\nu}$.  Moreover, the modules $V_{\mu,\nu}$ are pairwise nonisomorphic (for distinct pairs of Young diagrams) and exhaust  (up to isomorphism) all simple objects of $\bT_{\fgl^M}$.

In {\cite{Chi14}} the simplicity of $V^*/V_*$ as a $\fgl^M$-module has been shown.  This yields three obvious simple modules in $\bT^3_{\fgl^M}$: $V$, $V_*$, and $V^*/V_*$.  A next important result of {\cite{Chi14}} is

\begin{theorem}[{{\cite{Chi14}}}]  For a triple of Young diagrams $\lambda,\mu,\nu$, the $\fgl^M$-module 
$$V_{\lambda,\mu,\nu}:=\left(V^*/V_*\right)_\lambda\otimes V_{\mu,\nu}$$
is simple.
\end{theorem}

In addition, the socle filtrations over $\fgl^M$ of the simple objects in $\bT^2_{\End_\bK(V)}$ are computed in {\cite{Chi14}}.  We will recall the result in \Cref{subse.inj_soc}.

We conclude this background section by a diagram of categories which helps to better understand our object of study.  Let $\bT^1_{\fgl^M}$ denote the finite-length tensor category with single generator $V$: its objects are finite direct sums of simples of the form $V_\lambda$ for arbitrary Young diagrams $\lambda$.  By $\bT^2_{\fgl^M}$ we denote the category $\bT_{\fgl^M}$, and by $\bT^2_{\End_\bK(V)}$ we denote the category $\bT_{\End_\bK(V)}$ {\cite{PS2}}.  We then have the following natural diagram of inclusions and a restriction functor
\begin{equation}\label{eq.diagram}
	\begin{tikzpicture}
  \matrix (m) [matrix of math nodes,row sep=3em,column sep=2em,minimum width=2em] {
    \bT^1_{\fgl^M} &  \bT^2_{\fgl^M}	&  \bT^3_{\fgl^M} &  \bT^4_{\fgl^M} \\
   & & \bT^2_{\End_\bK(V)}\,,& \\
  };
 \path[-stealth, auto] (m-1-1) edge[draw=none]
                                    node [sloped, auto=false,
                                     allow upside down] {$\subset$}                 (m-1-2)
                               (m-1-2)  edge[draw=none]
                                    node [sloped, auto=false,
                                     allow upside down] {$\subset$}            (m-1-3)
                              (m-1-3) edge[draw=none]
                                    node [sloped, auto=false,
                                     allow upside down] {$\subset$}      (m-1-4)
                       (m-2-3) edge[decorate,decoration={snake,amplitude=.4mm,segment length=2mm,post length=1mm}] node {$r$}	(m-1-3);
\end{tikzpicture}
\end{equation}
%
%\begin{equation}
%    \begin{tikzpicture}[auto,baseline=(current  bounding  box.center)]
%    \path[anchor=base] (0,0) node (1) {$\bT^1_{\fgl^M}$} +(3,0) node (2) {$\bT^2_{\fgl^M}$} +(6,0) node (3) {$\bT^3_{\fgl^M}$} +(9,0) node (4) {$\bT^4_{\fgl^M}$} + (6,-2) node (5) {$\bT^2_{\End_\bK(V)}$};
%         \draw[->] (1) to node[auto,pos=.5] {$\subset$} (2);
%         \draw[->] (2) to node[auto,pos=.5] {$\subset$} (3);
%         \draw[->] (3) to node[auto,pos=.5] {$\subset$} (4);
%         \draw[->,decorate,decoration={snake,amplitude=.4mm,segment length=2mm,post length=1mm}] (5) to node[auto,pos=.5,swap] {$r$} (3);
%  \end{tikzpicture}
%\end{equation}
%
%\begin{equation}
%	\begin{tikzcd}
%				\bT^1_{\fgl^M\left(V,V_*\right)}\arrow{r}{\subset}
%				&\bT^2_{\fgl^M\left(V,V_*\right)} \arrow{r}{\subset}
%				&\bT^3_{\fgl^M\left(V,V_*\right)} \arrow{r}{\subset}
%				& \bT^4_{\fgl^M\left(V,V_*\right)} 
%				\\
%				& 	
%				& 
%				\bT^2_{\End(V)}\arrow[CurvyArrow]{u}[right]{\phantom{a}r}\,,
%				&
%			\end{tikzcd}
%	\label{eq.diagram}
%\end{equation}
%
%
%
where the category $\bT^4_{\fgl^M}$ is defined by adjoining the $\fgl^M$-module $\left(V_*\right)^*$ to $\bT^3_{\fgl^M}$, see \Cref{sec.four}.  All categories \Cref{eq.diagram} are finite-length tensor categories.  The structure of the categories $\bT^1_{\fgl^M}$, $\bT^2_{\fgl^M}$, and $\bT^2_{\End_\bK(V)}$ is well understood, in particular, the latter two categories are canonically equivalent to the category $\bT_{\fsl\left(V,V_*\right)}$ studied in~\cite{DPS,SS}.  In the present paper we investigate mainly the category $\bT^3_{\fgl^M}$.  %A key idea of the present paper is to develop the formalism of (partially) ordered Grothendieck categories, which we then apply to establish results on the structure of the category $\bT^3_{\fgl^M}$. 
 %The category $\bT^4_{\fgl^M}$ is not an ordered tensor category, and in this paper 
For $\bT^4_{\fgl^M}$, we only establish preparatory results: we prove that $\bT^4_{\fgl^M}$ is a finite-length category and we classify its simple objects.

%%%%%%%%%%%%%%%%%%%%%%%%%%%%%%%%%%%%%%%%%%%%%%%%%%%%%%%%%%%%%%%%%%%%%%%%%%%%%%%%%%%%%%%%%%%%%%%%%%%%%%%%%%%%%%%%%%
%%%%%%%%%%%%%%%%%%%%%%%%%%%%%%%%%%%%%%%%%%%%%%%%%%%%%%%%%%%%%%%%%%%%%%%%%%%%%%%%%%%%%%%%%%%%%%%%%%%%%%%%%%%%%%%%%%
\section{Ordered Grothendieck categories% and injective hulls
}\label{se.inj}

\subsection{Definition and characterization of indecomposable injectives}
\label{subse.indinj}

%The setting for this section is as follows. 

We now define a class of Grothendieck categories which we study throughout \Cref{se.inj}.

Let $\mathcal{C}$ be a $\bK$-linear Grothendieck category (for the notion of Grothendieck category we refer the reader to \cite{groth}). Unless specified otherwise, all additive categories are assumed $\bK$-linear and all functors between $\bK$-linear categories are understood to be $\bK$-linear.  The sign $\subset$ denotes a monomorphism in $\mathcal{C}$, or a not necessarily strict set-theoretic inclusion.  If $Z$ is an object $\mathcal{C}$, then $Z^{\oplus q}$ stands for the direct sum of $q$ copies of $Z$.  

Let $X_i\in\mathcal C$, $i\in I$ be objects indexed by a partially ordered set $(I,\leq)$ with the property that every element dominates finitely many others; in other words, for every $i\in I$ the down-set
\begin{equation*}
  I_{\le i} := \{j\in I\ |\ j\le i\}
\end{equation*}
is finite. For each $i\in I$, we fix a finite set $\Theta_i$ of morphisms from $X_i$ into various objects $X_j$ for $j<i$.  Denote \[ Y_i:=\bigcap_{f\in\Theta_i}\mathrm{ker}\,f. \] We further assume that 

\renewcommand{\labelenumi}{(\arabic{enumi})}
\begin{enumerate}
\item every object of $\mathcal C$ is a sum of subquotients of finite direct sums of $X_i$;
\item $Y_i$ has finite length, and $Y_i=\soc X_i$; write $\displaystyle Y_i=\bigoplus_{S\in \cS_i}S^{\oplus p(S)}$, where $\cS_i$ is the set of isomorphism classes of distinct simple direct summands of $Y_i$, and $p(S)$ is the multiplicity of $S$ in $Y_i$;
\item the sets $\cS_i$ are disjoint;
\item $X_i$ decomposes as $\displaystyle \bigoplus_{S\in \cS_i}\widetilde S^{\oplus p(S)}$ for modules $\widetilde S$ with socle $S$. 
\end{enumerate}
In the sequel we will sometimes refer to a simple object as belonging to $\cS_i$; by this we mean that its isomorphism class belongs to that set.

\begin{definition}
  An \define{ordered Grothendieck category} is a category $\cC$ satisfying all above assumptions. \end{definition}

\begin{remark}
  Conditions (1) through (4) ensure that all $X_i$ have finite length, and hence so do subquotients of finite direct sums of $X_i$. In particular, we can freely use the Jordan-H\"older Theorem for such objects.  
\end{remark}

The definition of ordered Grothendieck category allows the following quick characterization of indecomposable injectives.

\begin{proposition}\label{pr.inj_env}
  If $\cC$ is an ordered Grothendieck category then $\widetilde S$ is an injective hull of $S$ in $\mathcal C$ for any $i\in I$ and any simple $S\in \cS_i$. 
\end{proposition}
%\begin{proof}

It suffices to show that a subquotient $J$ of a finite direct sum of $X_i$ admitting an essential extension $S\subset J$ in fact embeds into $\widetilde S$. The following lemma deals with a particular case of this situation. 

\begin{lemma}\label{le.part}
Let $S\subset J$ be an essential extension.  If $J$ is a subquotient of a direct sum $X_i^{\oplus q}$ for some $i$ and some $q$, then $J$ is isomorphic to a subobject of $\widetilde S$. 
\end{lemma}
\begin{proof}
Subquotient means subobject of a quotient. So let $\pi:X_i^{\oplus q}\to Z$ be an epimorphism such that $J\subset Z$. We regard $S$ as a subobject of $\soc Z$ by means of the monomorphisms $S\subset J\subset Z$, and there is a decomposition $\soc Z=T\oplus S$. Since the extension $S\subset J$ is essential, $T$ intersects $J$ trivially, and hence, after factoring out $T$, we can (and will) assume that $Z$ has simple socle $S$.

Now, the socle $Y_i^{\oplus q}$ of $X_i^{\oplus q}$ can be decomposed as $U\oplus S$ in such a manner that $\pi|_{Y_i^{\oplus q}}$ is the projection on the second direct summand. By condition (4) above, this decomposition can be lifted to a decomposition of $X_i^{\oplus q}$ as $\widetilde U\oplus\widetilde S$. The socle $U$ of $\widetilde U$ is already in the kernel of $\pi$, and all other simple constituents of $\widetilde U$ belong to the set $\bigcup_{j<i}\mathcal{S}_j$ which does not contain $S$. It follows that $\pi(\widetilde U)$ intersects $S$ trivially, and hence also $J$.  Consequently, $J$ admits a monomorphism into $Z/\pi(\widetilde U)$. But $Z/\pi(\widetilde{U})$ is a quotient of $\widetilde S$, and it can only be isomorphic to $\widetilde S$ as the restriction of $\pi$ to the simple socle $S$ of $\widetilde S$ is a monomorphism. We are done.   
\end{proof}

\begin{proof_of_inj_env}
Let $\pi: X\to J$ be an epimorphism, for some subobject $\displaystyle X\subset\bigoplus_{j\in I} X_j^{\oplus q_j}$ of a finite direct sum. Any simple subobject $T$ of $X$ belonging to $S_j$ for $j\ne i$ will automatically lie in the kernel of $\pi$ (by condition (3) above), so we may as well assume that $X$ is a subobject of $X_i^{\oplus q_i}$. Then we apply \Cref{le.part}.
\end{proof_of_inj_env}

\begin{corollary}~\label{cor.indinjobj}
	The indecomposable injective objects in $\mathcal{C}$ are isomorphic to arbitrary indecomposable direct summands of the objects $X_i$.
\end{corollary}

%Note that this is already enough to prove the relevant results about $\mathbb{T}_{\mathfrak{so}}$, $\mathbb{T}_{\mathfrak{sp}}$ and $\mathbb{T}_{\mathfrak{gl}}$: universality property, which implies iso for the $\mathfrak{so}$ and $\mathfrak{sp}$), as well as $\mathrm{End}(T(V))$ or $\mathrm{End}(T(V\oplus V_*))$.   

%
%%%%%%%%%%%%%%%%%%%%%%%%%%%%%%%%%%%%%%%%%%%%%%%%%%%%%%%%%%%%%%%%%%%%%%%%%%%%%%%%%%%%%%%%%%%%%%%%%%%%%%%%%%%%%%%%%%%
\subsection{Bounds on non-vanishing ext functors}
\label{subse.ext}

For $i\le j\in I$ define the \define{defect} $d(i,j)$ to be the number of links (i.e. symbols `$<$') in a longest chain $i<\cdots<j$. Note that, by convention, $d(i,j)$ is not defined unless $i\le j$.

\begin{remark}
  It follows easily from the definition that, for triples $i\le j\le k$, the function $d$ satisfies the opposite of the triangle inequality, i.e. $d(i,k)\ge d(i,j)+d(j,k)$.  
\end{remark}

Let $S\in \cS_s$ and $T\in \cS_t$ be simple objects in an ordered Grothendieck category $\cC$.

\begin{lemma}\label{le.ext_1st_approx}
If $\Ext^p(S,T)\ne 0$ then $s\le t$. 
\end{lemma}
\begin{proof}
We proceed by induction on $p$, the base case $p=0$ being immediate. 

Suppose $p>0$ and the statement holds for all smaller $p$.  Let $\widetilde{T}$ be an injective hull of $T$. The long exact sequence
\begin{equation*}
\Ext^{p-1}(S,T)\to \Ext^{p-1}(S,\widetilde{T})\to \Ext^{p-1}(S,\widetilde{T}/T)\to \Ext^p(S,T)\to \Ext^p(S,\widetilde T)\to\cdots\end{equation*}
identifies $\Ext^p(S,T)$ with $\Ext^{p-1}(S,\widetilde{T}/T)$. Indeed, this is clear if $p>1$ because then both the second and fifth terms in the sequence are zero. On the other hand, when $p=1$  the leftmost arrow is an epimorphism because $T$ is the socle of $\widetilde{T}$, and hence the second leftmost arrow is zero. 

The conclusion follows from the induction hypothesis and the fact that the quotient $\widetilde{T}/T$ has a filtration whose successive quotients belong to sets $\cS_{t'}$ for $t'<t$. 
\end{proof}

%
%
%Now suppose the poset $I$ is \define{ranked}, i.e. equipped with a poset map $r:I\to \bN = \{0,1,\ldots\}$, where the latter is endowed with the usual order. We refer to $r(i)$ as \define{the rank of $i$}. For convenience, we assume $I$ has a smallest element $*$ such that the zero object of $\cC$ is the only element of $\cS_*$, and set $r(*)=-1$. 
%
%
%\begin{definition}\label{def.step_1}
%A morphism $f:X_i\to X_j$ with $j<i\in I$ is \define{short} if $r(i)-r(j)=1$.  
%\end{definition}
%
%
%
%The proof consists of successive auxiliary lemmas. 
%
%
%
We can now improve on this somewhat, leading to the main result of this subsection. Recall that $d(i,j)$ was only defined for $i\le j$, so \Cref{le.ext_1st_approx} is necessary for the statement below to make sense.

\begin{proposition}\label{pr.ext}
$\Ext^p(S,T)\ne 0$ implies $d(s,t)\ge p$. 
\end{proposition}
\begin{proof}
To prove the inequality for all $p$ we can once more perform induction on $p$ with the case $p=0$ being obvious. Assume now that $p$ is positive and that the induction hypothesis is in place. 

The same long exact sequence that we used in the proof of \Cref{le.ext_1st_approx} implies
\begin{equation*}
  \Ext^p(S,T)\cong \Ext^{p-1}(S,\widetilde{T}/T).
\end{equation*}  
This means that $\Ext^{p-1}(S,T')$ is nonzero for some simple $T'\in \cS_{t'}$, $t'<t$, and the induction hypothesis then ensures that $d(s,t)>d(s,t')\ge p-1$. 
\end{proof}

We will later need the following variant of \Cref{pr.ext} for $\Ext^1$. Before we state it, a bit of terminology.

\begin{definition}\label{def.short}
  A morphism $f:X_i\to X_j$ in $\cC$ is \define{short} if $d(j,i)=1$. 
\end{definition}

\begin{lemma}\label{le.ext1}
Assume furthermore that all morphisms $f\in \Theta_i$ are short. In that case, $\Ext^1(S,T)\ne 0$ implies $d(s,t)=1$.  
\end{lemma}
\begin{proof}
  We already know from \Cref{le.ext_1st_approx} and \Cref{pr.ext} that $s\le t$ and $d(s,t)\ge 1$. The fact that the strict inequality $d(s,t)>1$ is impossible follows from the observation that all nontrivial extensions of a simple object by $T$ are subobjects of the injective hull $\widetilde{T}$ of $T$, and by assumption the socle of $\widetilde{T}/T$ is a direct sum of simples $T'\in \cS_{t'}$ with $t'<t$ and $d(t',t)=1$.  
\end{proof}

%%%%%%%%%%%%%%%%%%%%%%%%%%%%%%%%%%%%%%%%%%%%%%%%%%%%%%%%%%%%%%%%%%%%%%%%%%%%%%%%%%%%%%%%%%%%%%%%%%%%%%%%%%%%%%%%%%%

\subsection{$\cC$ as a comodule category}
\label{subse.comod}

We are now ready to characterize ordered Grothendieck categories as certain categories of comodules.    Given an ordered Grothendieck category $\mathcal{C}$ we denote by $\mathcal{C}_{\fin}$, the full, thick abelian subcategory of $\mathcal{C}$ generated by $X_i$ (i.e., the minimal such category containing $X_i$ for $i\in I$).

\begin{definition}
  A coalgebra $C$ is \define{left semiperfect} if every indecomposable injective \define{right} $C$-comodule is finite dimensional. 
\end{definition}

\begin{remark}
  This is not quite the standard definition. By analogy with the dual notion for algebras, the requirement is that every finite-dimensional \define{left} $C$-comodule have a projective cover. However, \cite[Theorem 10]{Lin77} ensures that the two conditions are equivalent. 
\end{remark}

For a coalgebra $C$ we denote by $\cM^C$ the category of right $C$-comodules, and by $\cM^C_{\fin}$ the category of finite-dimensional right comodules.% Finally, let $\cC_{\fin}$ be the full, thick abelian subcategory of $\cC$ generated by $X_i$. 

\begin{theorem}\label{th.comod}
  Suppose $\mathcal{C}$ is an ordered Grothendieck category such that the endomorphism ring of any simple $\cC$-objects is finite dimensional over $\bK$. Then there is a $\bK$-coalgebra $C$ and an equivalence $\cC\overset{\sim}{\rightsquigarrow} \cM^C$ of $\bK$-linear categories. 

Moreover, any such coalgebra $C$ is left semiperfect, and any such equivalence identifies $\cC_{\fin}$ and $\cM^C_{\fin}$. 
\end{theorem}
%
%
%a commutative square 
%\begin{equation*}
%  \begin{tikzpicture}[auto,baseline=(current  bounding  box.center)]
%    \path[anchor=base] (0,0) node (1) {$\cC$} +(2,0) node (2) {$\cM^C$} +(0,-1.5) node (3) {$\cC_{\fin}$} +(2,-1.5) node (4) {$\cM^C_{\fin}$};
%         \draw[->] (1) to node[auto,pos=.5] {$\simeq$} (2);
%         \draw[->] (3) to node[auto,pos=.5] {$\simeq$} (4);
%         \path[-stealth, auto] (3) edge[draw=none]
%             node [sloped, auto=false,
%              allow upside down] {$\subset$} (1);
%         \path[-stealth, auto] (4) edge[draw=none]
%             node [sloped, auto=false,
%              allow upside down] {$\subset$} (2);
%  \end{tikzpicture}
%\end{equation*}
%of inclusions and equivalences.
%
%
\begin{proof}
According to \cite[Definitions 4.1 and 4.4, Theorem 5.1]{Tak77}, in order to prove the first assertion it suffices to check that $\cC$ has a set of generators of finite length. This is immediate: simply take the set consisting of all subquotients of finite direct sums of $X_i$s. 

The existence of a $\bK$-linear equivalence of categories $\cC \overset{\sim}{\rightsquigarrow} \cM^C$ forces $C$ to be left semiperfect, as the indecomposable injectives of $\cC$ are of finite length (\Cref{pr.inj_env})
%are precisely the indecomposable summands of the $X_i$, and these are of finite length 
and hence correspond to finite-dimensional $C$-comodules. 

We now prove the last assertion that any $\bK$-linear equivalence $\cC \overset{\sim}{\rightsquigarrow} \cM^C$ automatically identifies $\cC_{\fin}$ and $\cM^C_{\fin}$. Note first that $\cC_{\fin}$ consists of those objects that are subquotients of finite direct sums of $X_i$s. By \Cref{cor.indinjobj}, the indecomposable direct summands of the $X_i$s are (up to isomorphism) precisely the indecomposable injectives in $\cC$. In general, comodule categories admit injective hulls, and an object is of finite length if and only if its injective hull is a finite direct sum of indecomposable injectives. It follows from this that the objects of $\cC_{\fin}$ are the finite-length objects in $\cC$. In turn, for $C$-comodules, being of finite length is equivalent to being finite dimensional.
\end{proof}

 In the course of the proof of \Cref{th.comod} we have obtained the following result in passing.

\begin{corollary}
	The category $\mathcal{C}_{\fin}$ consists of all objects of finite length in $\mathcal{C}$.
\end{corollary}
%\begin{proof}
%  This follows from \Cref{th.comod} and the fact that every $C$-comodule is a sum of finite-dimensional sub-comodules (\cite[Theorem 2.1.3]{Sweedler}), and hence $\cM^C_{\fin}\subset \cM^C$ is the full subcategory on the finite-length objects.  
%\end{proof}

%%%%%%%%%%%%%%%%%%%%%%%%%%%%%%%%%%%%%%%%%%%%%%%%%%%%%%%%%%%%%%%%%%%%%%%%%%%%%%%%%%%%%%%%%%%%%%%%%%%%%%%%%%%%%%%%%%%
\subsection{$\cC$ as a highest weight category}
\label{subse.hw}

We show next that an ordered Grothendieck category is a highest weight category in the sense of \cite[Definition 3.1]{CPS}.   First, recall the definition from loc. cit.

\begin{definition}\label{def.hw}
  A $\bK$-linear category $\cC$ is a {\it highest weight category} if it is locally Artinian and there exists an interval-finite partially ordered set $\Lambda$ such that
\renewcommand{\labelenumi}{(\arabic{enumi})}
  \begin{enumerate}
    \item There exists a complete collection $\{S(\lambda)\}_{\lambda\in \Lambda}$ of simple objects in $\cC$.
    \item There is a collection $\{A(\lambda)\}_{\lambda\in \Lambda}$ of objects admitting embeddings $S(\lambda)\subset A(\lambda)$ such that all composition factors $S(\mu)$ of $A(\lambda)/S(\lambda)$ satisfy $\mu<\lambda$. Moreover, for all $\lambda,\mu\in \Lambda$ both $\dim_\bK\Hom_\cC(A(\lambda),A(\mu))$ and the multiplicity $[A(\lambda):S(\mu)]$ are finite. 
    \item Each simple object $S(\lambda)$ has an injective hull $I(\lambda)$, and the latter has a good filtration $0=F_0(\lambda)\subset F_1(\lambda)\subset\cdots$.
  \end{enumerate}
\end{definition}
We will not recall the definition of good filtration here.  We just mention that in a good filtration one reqires that $F_1(\lambda)\cong A(\lambda)$ and that the other successive subquotients $F_{n+1}(\lambda)/F_n(\lambda)$ be isomorphic to various $A(\mu)$, $\mu>\lambda$.

\begin{proposition}\label{pr.hw}
  An ordered Grothendieck category is a highest weight category. 
\end{proposition}
\begin{proof}
  Let $\cC$ be an ordered Grothendieck $\bK$-linear category, and $I$, $X_i$, $Y_i$, $S_i$ be as in \Cref{subse.indinj}.  It is well known that comodule categories are locally Artinian, so $\cC$ is locally Artinian by \Cref{th.comod}. %We keep the rest of the notation used thus far, whereby we have objects $X_i\in \cC$, sets $\cS_i$ of (isomorphism classes of) simple objects, etc. 

Set $\Lambda:=\bigcup_{i\in I}\cS_i$, and regard it as an ordered set by declaring $\lambda< \mu$ precisely when $\lambda\in \cS_i$ and $\mu\in \cS_j$ with $i<j$. Because $I$ has finite down-sets, $(\Lambda,\le)$ clearly has finite intervals. By definition, the poset $\Lambda$ indexes the complete set $\bigcup \cS_i$ of isomorphism classes of simple objects in $\cC$, taking care of part (1) of \Cref{def.hw}. 

Now take $A(\lambda)$ to be the injective hull of the simple object $S(\lambda)$ (hence, according to \Cref{cor.indinjobj}, $A(\lambda)$ is a summand of $X_i$ if $\lambda\in \cS_i$). The finiteness conditions in part (2) are satisfied because the objects $X_i$ have finite length. The condition that $A(\lambda)/S(\lambda)$ admits a filtration with subquotients $S(\mu)$, $\mu<\lambda$ follows from the fact that $S(\lambda)\subset A(\lambda)$ is the kernel of a map of $A(\lambda)$ into a sum of $A(\mu)$ for $\mu<\lambda$.  This latter fact is a direct corollary of \Cref{pr.inj_env} and the definition of $Y_i$.  %See the initial discussion on ordered Grothendieck categories, at the beginning of \Cref{se.inj}. 

Finally, the sought-after good filtration of the injective hull $I(\lambda)\supset S(\lambda)$, in our case, is as small as possible: $0=F_0(\lambda)\subset A(\lambda)=I(\lambda)$. Because here the filtration is so degenerate, the other properties required of good filtrations hold vacuously. 
\end{proof}

We end this brief subsection by noting that our main motivation for introducing the formalism of ordered Grothendieck categories was that it allows a relatively quick explicit characterization of the indecomposable injective objects in the category.  In a general highest weight category the indecomposable injectives are not even required to have finite length, so ordered Grothendieck categories are a rather special class of highest weight categories.  %proving that certain objects in the category are injective, and hence that is where most of the effort was concentrated, in the above discussion. 

%On the other hand the definition of a highest weight category assumes a certain structure on injective hulls of simple objects as a starting point; for this reason, the two concepts have a somewhat different focus.   

%%%%%%%%%%%%%%%%%%%%%%%%%%%%%%%%%%%%%%%%%%%%%%%%%%%%%%%%%%%%%%%%%%%%%%%%%%%%%%%%%%%%%%%%%%%%%%%%%%%%%%%%%%%%%%%%%%%

\subsection{Universal properties for $\cC$ and $\cC_{\text{\tiny{FIN}}}$}
\label{subse.univ}

We henceforth assume that all endomorphism rings of simple objects in $\cC$ are equal to $\bK$. Together with our already standing assumptions, this ensures that the categories $\cC$ and $\cC_{\fin}$ are universal in a certain sense that we will make precise shortly. 

Denote by $\cC_X$ the smallest subcategory of $\cC$ containing the $X_i$ as objects and the $\Theta_i$ as sets of morphisms, and closed under taking $\bK$-linear combinations, compositions, finite direct sums and direct summands of morphisms.

\begin{theorem}\label{th.univ}
Let $\cD$ be an abelian category and $F:\cC_X\rightsquigarrow \cD$ be a $\bK$-linear functor. 
\begin{enumerate}
\renewcommand{\labelenumi}{(\alph{enumi})}
\item $F$ extends to a left-exact functor $\cC_{\fin}\rightsquigarrow \cD$ uniquely up to natural isomorphism.  
\item If in addition $\cD$ has arbitrary coproducts, $F$ extends uniquely (up to natural isomorphism) to a left-exact, coproduct-preserving functor $\cC\rightsquigarrow \cD$. 
\end{enumerate}
\end{theorem}

We first need the following observation.

\begin{lemma}\label{le.CX_full}
The subcategory $\cC_X\subset \cC$ is full on the finite direct sums of indecomposable injectives in $\cC$. 
\end{lemma}
\begin{proof}
Let $\widetilde{S}_i\subset X_i$ and $\widetilde{S}_j\subset X_j$ be two indecomposable injectives  for $i,j\in I$, and $f:\widetilde{S}_i\to \widetilde{S}_j$ be a morphism in $\cC$. We will prove that $f$ is a morphism in $\cC_X$ by induction on the defect $d(i,j)$,% under the assumption that this claim is true for smaller defects (
with the convention that if $d(i,j)=0$ then this assumption is true vacuously. 

If $f$ is a monomorphism, then it is an isomorphism. Moreover, $f$ is unique up to scaling by our assumption that the $\cC$-endomorphism rings of simple objects equal $\bK$, so $f$ must indeed be a morphism in $\cC_X$ as defined above.

Now suppose $f$ is not a monomorphism. Then $f$ annihilates the socle $S_i$ of $\widetilde{S}_i$, and hence can be thought of as a morphism from the quotient $\widetilde{S}_i/S_i$ to $\widetilde{S}_j$.  By assumption, this quotient admits a monomorphism into a finite direct sum of $\widetilde{S}_k$ (with $k<i$ in $I$) via morphisms in $\Theta_i$, and $f:\widetilde{S}_i/S_i\to \widetilde{S}_j$ extends to the direct sum of the $\widetilde{S}_k$ by the injectivity of $\widetilde{S}_j$. Since all $k$ are less than $i$, the conclusion follows from the induction hypothesis. 
\end{proof}

\begin{remark}
  Note that the above proof shows that there are no nonzero morphisms $X_i\to X_j$ for $i<j$. 
\end{remark}

\begin{proof_of_univ}
{\bf (a)} %This is presumably well known, but we include a proof as we could not find this precise statement in the literature. 
By \Cref{le.CX_full} we know that $\cC_X$ is the full subcategory of $\cC_{\fin}$ on the injective objects in $\cC_{\fin}$. Since the lower-bounded derived category $D^+(\cC_{\fin})$ can be built out of complexes of injectives, $F$ extends to a functor $DF:D^+(\cC_{\fin})\rightsquigarrow D^+(\cD)$ of triangulated categories. 

Now consider the standard t-structures $(D^{\le 0},D^{\ge 0})$ for both $D=D^+(\cC_{\fin})$ and $D=D^+(\cD)$. Since $DF$ is clearly left t-exact in the sense that $DF(D^+(\cC_{\fin})^{\ge 0})\subset D^+(\cD)^{\ge 0}$, the functor restriction of $H^0(DF):D^+(\cC_{\fin})\rightsquigarrow \cD$ to the heart $\cC_{\fin}$ of the standard t-structure will be left exact (see e.g. \cite[Proposition 8.1.15]{HTT08}). It is moreover easily seen to extend $F$. 

{\bf (b)} $F$ extends uniquely in a coproduct-preserving fashion to the full subcategory of $\cC$ on \define{arbitrary} direct sums of indecomposable injectives. Since these are all of the injectives in $\cC$ (a consequence of the fact that $\cC$ is a comodule category by \Cref{th.comod}), the argument from part (a) then extends $F$ uniquely as a left exact, coproduct-preserving functor into $\cD$ as desired. 
\end{proof_of_univ}

%%%%%%%%%%%%%%%%%%%%%%%%%%%%%%%%%%%%%%%%%%%%%%%%%%%%%%%%%%%%%%%%%%%%%%%%%%%%%%%%%%%%%%%%%%%%%%%%%%%%%%%%%%%%%%%%%%%
\subsection{The monoidal case}
\label{subse.tensor}

Recall that all of our additive categories are assumed $\bK$-linear.

\begin{definition}\label{def.tensor_Groth}
  An additive category is \define{monoidal} if it has a monoidal structure such that all functors of the form $x\otimes\bullet$ and $\bullet\otimes x$ are exact. 
 
An \define{additive tensor category} (or just tensor category for short) is an additive category that is monoidal in the above sense and for which in addition its monoidal structure is symmetric. 

A \define{tensor functor} is a symmetric monoidal $\bK$-linear functor between tensor categories. 
\end{definition}

Now, in the setting of \Cref{th.univ}, suppose $\cC$ is a tensor category. Suppose furthermore that the set $\{X_i\}$ and the linear span of $\bigcup \Theta_i$ are closed under tensor products. Then $\cC_{\fin}$ is a tensor category. 

Finally, denote by $\cC_{X,\otimes}$ the smallest tensor subcategory containing $X_i$, $\Theta_i$, and closed under direct summands. Adapting the proof of \Cref{th.univ} in a routine fashion we get

\begin{theorem}\label{th.tens_univ}
Let $\cD$ be a tensor abelian category and $F:\cC_{X,\otimes}\rightsquigarrow \cD$ a $\bK$-linear tensor functor. 
\begin{enumerate}
\renewcommand{\labelenumi}{(\alph{enumi})}
\item $F$ extends to a left exact tensor functor $\cC_{\fin}\rightsquigarrow \cD$ uniquely up to the obvious notion of tensor natural isomorphism.  
\item If in addition $\cD$ has arbitrary coproducts, $F$ extends uniquely to a left exact, coproduct-preserving tensor functor $\cC\rightsquigarrow \cD$. \qedhere
\end{enumerate}
\end{theorem}

%%%%%%%%%%%%%%%%%%%%%%%%%%%%%%%%%%%%%%%%%%%%%%%%%%%%%%%%%%%%%%%%%%%%%%%%%%%%%%%%%%%%%%%%%%%%%%%%%%%%%%%%%%%%%%%%%%%%
\subsection{An application: the category $\bT^2_{\fgl^M}$}
\label{subse.appl}

We illustrate the usefulness of the preceding material with an application to the category $\bT^2_{\fgl^M}$ introduced in \Cref{se.prel}. The setup for doing so is as follows. 

 We regard $\bT^2_{\fgl^M}$ as $\cC_{\fin}$, where $\cC$ is the Grothendieck category consisting of objects that are (possibly infinite) sums of objects in $\bT^2_{\fgl^M}$. The poset $I$ is the set of ordered pairs $(m,n)$ of nonnegative integers, and the order on $I$ is the smallest partial order such that $(m,n)\ge (m-1,n-1)$ for positive integers $m$ and $n$. The object $X_i$ corresponding to the element $i\in (m,n)$ in $I$ is by definition $(V_*)^{\otimes m}\otimes V^{\otimes n}$. 

Next, for $i=(m,n)$, we define $\Theta_i$ to be the set of contractions 
\begin{equation*}
  \mathrm{p}^{m,n}_{r,s}:(V_*)^{\otimes m}\otimes V^{\otimes n}\to   (V_*)^{\otimes (m-1)}\otimes V^{\otimes (n-1)}\,,
\end{equation*}
where
\begin{align*}
\mathrm{p}^{m,n}_{r,s}\left(y_1\otimes \cdots \otimes y_m\otimes x_1 \otimes \cdots\otimes x_n\right)\phantom{p_{r,s}\left(y_1\otimes \cdots \otimes y_m\otimes x_1 \otimes \cdots\otimes x_n\right)p\left(y_r\otimes x_s\right)\,y_1\otimes\cdots\otimes }
\\=\mathrm{p}\left(y_r\otimes x_s\right)\,y_1\otimes\cdots\otimes y_{r-1}\otimes y_{r+1}\otimes \cdots \otimes y_m
\otimes x_1\otimes \cdots \otimes x_{s-1}\otimes x_{s+1}\otimes \cdots \otimes x_n\,.
\end{align*}  The fact that conditions (1) -- (4) from the definition of an ordered Grothendieck category for this choice of $X_i$ and $\Theta_i$ hold follows from results of \cite{PS}.  Note also that the contractions $\mathrm{p}_{r,s}^{m,n}$ are short morphisms in the sense of \Cref{def.short}.

With this in place, \Cref{pr.inj_env} gives now an alternate proof for the injectivity of the objects $(V_*)^{\otimes m}\otimes V^{\otimes n}$ in $\bT^2_{\fgl^M}$. Cf. \cite{DPS,SS} for different approaches to injectivity.

%%%%%%%%%%%%%%%%%%%%%%%%%%%%%%%%%%%%%%%%%%%%%%%%%%%%%%%%%%%%%%%%%%%%%%%%%%%%%%%%%%%%%%%%%%%%%%%%%%%%%%%%%%%%%%%%%%
%%%%%%%%%%%%%%%%%%%%%%%%%%%%%%%%%%%%%%%%%%%%%%%%%%%%%%%%%%%%%%%%%%%%%%%%%%%%%%%%%%%%%%%%%%%%%%%%%%%%%%%%%%%%%%%%%%
\section{The three-diagram category}\label{se.3d}

%{\color{blue}Note: Not a correction.  The aim should be to describe the simple objects of $\bT^3$.}

%%%%%%%%%%%%%%%%%%%%%%%%%%%%%%%%%%%%%%%%%%%%%%%%%%%%%%%%%%%%%%%%%%%%%%%%%%%%%%%%%%%%%%%%%%%%%%%%%%%%%%%%%%%%%%%%%%
\subsection{Simple objects}

%Recall the Lie algebra $\fg^M=\fgl(V,V_*)$ introduced in \Cref{se.prel}. 

We now begin our study of the category $\bT^3_{\fgl^M}$, see \Cref{se.prel} for the definition.  The aim of this subsection is to describe the simple objects of 
$\bT^3_{\fgl^M}$.  
We first show that the simple $\fg^M$-modules $(V^*/V_*)_\lambda\otimes V_{\mu,\nu}$ from \cite{Chi14} are mutually nonisomorphic as $(\lambda,\mu,\nu)$ ranges over all ordered triples of Young diagrams.% and $\lambda$ ranges over all Young diagrams. %This will be important when applying the results of \Cref{se.inj}.

We start with the following variant of \cite[Lemma 3]{Chi14}. Before stating it, recall that a Lie subalgebra $I$ of a Lie algebra $\fg$ acts \define{densely} on a $\fg$-module $W$ if for any finite set $\{w_i\}_{i=1}^n\subset W$ and any $g\in \fg$ there is some $g'\in I$ such that 
\begin{equation*}
  g'w_i = gw_i\,\, \text{for any}\ i,\  1\le i\le n.
\end{equation*}

\begin{lemma}\label{le.hom_split}
  Let $\fG$ be a Lie algebra and $I\subseteq \fG$ be an ideal. Let $U$ and $U'$ be $\fG/I$-modules, and $W$ be a $\fG$-module on which $I$ acts densely and irreducibly with $\End_I(W)=\bK$. Then 
  \begin{equation*}
    \Hom_{\fG}(U\otimes W,U'\otimes W) = \Hom_{\fG}(U,U').
  \end{equation*}
\end{lemma}
\begin{proof}
Let $f\in\Hom_{\fG}(U\otimes W,U'\otimes W)$. We will show that there is $g\in\Hom_{\fG}(U,U')$ such that $f=g\otimes\id_W$.

Fix $u\otimes w\in U\otimes W$ and let
\begin{equation*}
  f(u\otimes w)=\sum_{i=1}^n u_i'\otimes w_i,\ u_i'\in U',\ w_i\in W. 
\end{equation*}
As $f$ is a homomorphism of $I$-modules and $I$ annihilates $U$, $f$ maps $u\otimes W\cong W$ to the direct sum 
\begin{equation*}
 \bigoplus_{i=1}^n  u_i'\otimes W\cong W^{\oplus n}.  
\end{equation*}
Furthermore, $f$ composed with the projection to a fixed direct summand of $W^{\oplus n}$ is an endomorphism of $W$ sending $u\otimes w\in u\otimes W\cong W$ to $u'_i\otimes w_i$.  But since $\End_I(W)=\bK$, each $w_i$ must be a scalar multiple $t_iw$ of $w$.  This shows that
\begin{equation*}
  g:u\mapsto \sum t_iu_i'
\end{equation*}
is a well-defined linear map $U\to U'$ with $f=g\otimes\id_W$. 

It remains to show that $g$ is a $\fG$-module homomorphism; we do this by following the part of the proof of \cite[Lemma 3]{Chi14} that uses the density condition. 

Let $h\in \fG$ be an arbitrary element, and let $k\in I$ be such that $hw = kw$ (this is possible by density). We have
\begin{equation*}
  (h-k)(u\otimes w) = hu\otimes w+u\otimes hw-u\otimes kw = hu\otimes w
\end{equation*}
because $ku=0$. This implies
\begin{equation*}
  g(hu)\otimes w = f((h-k)(u\otimes w)) =(h-k)f(u\otimes w) = (h-k)(g(u)\otimes w) = hg(u)\otimes w. 
\end{equation*}
Since $w$ is arbitrary, we get the desired conclusion $hg(u) = g(hu)$ for all $u\in U$, $h\in\fG$, i.e., $g$ is a $\fG$-module homomorphism. 
\end{proof}

We also need the following complement to \cite[Proposition 1]{Chi14}; its proof is based on some of the same ideas as in loc. cit.

\begin{proposition}\label{pr.S-W}
  The endomorphism algebra of the $\fg^M$-module $(V^*/V_*)^{\otimes m}$ is isomorphic to the group algebra $\bK\left[S_m\right]$ of the symmetric group $S_m$. 
\end{proposition}
\begin{proof}
  The infinite dimensionality of $V^*/V_*$ makes it clear that the different permutations from $S_m$ are linearly independent as endomorphisms of $(V^*/V_*)^{\otimes m}$, so it remains to show that there are no other $\fg^M$-endomorphisms of $(V^*/V_*)^{\otimes m}$.  For this, it suffices to show that for distinct Young diagrams $\lambda$ and $\lambda'$ with $m$ boxes the simple $\fg^M$-modules $(V^*/V_*)_\lambda$ and $(V^*/V_*)_{\lambda'}$ are nonisomorphic.

As in \cite[$\S$ 2]{Chi14}, we regard the elements of $V^*$ as $\bZ_{>0}$-indexed row vectors with the following action of $\bZ_{>0}\times \bZ_{>0}$ matrices in $\fg^M$: $gv:=-vg$.
% act on the left by minus right multiplication
Then $V_*$ consists of such vectors with finitely many nonzero entries. Similarly, we think of $V^*/V_*$ as consisting of row vectors defined up to changing finitely many entries. 

Partition $\bZ_{>0}$ into $m$ infinite subsets $I_j$, $1\le j\le m$, and let $W_j$ be the subspace of $V^*/V_*$ consisting of vectors that have entries zero outside $I_j$. Let 
\begin{equation*}
  (V^*/V_*)_\lambda\ni x = \sum_{\sigma\in S_m}a_\sigma \left(x_{\sigma(1)}\otimes\cdots\otimes x_{\sigma(m)}\right) 
\end{equation*}
be some nonzero element, where $x_i\in W_i$. We can find an embedding $\fgl(m)\subset \fg^M$ sending the elementary matrix with entry $1$ in position $(i,j)$ and $0$ elsewhere to an element of $\fg^M$ which maps $W_j$ onto $W_i$ isomorphically so that $x_j$ gets mapped to $x_i$, and maps $W_{j'}$ to zero for $j'\neq j$.  

Now, suppose we have an isomorphism 
\begin{equation*}
  f:(V^*/V_*)_\lambda\to (V^*/V_*)_{\lambda'}
\end{equation*}
of $\fg^M$-modules. For each $1\le i\le m$, let $p_i\in \fg^M$ a matrix with zero entries outside $I_i\times I_i\subset  \bZ_{>0} \times  \bZ_{>0} $, such that
\begin{itemize}
  \item $p_i$ preserves the finite-dimensional subspace of $V^*/V_*$ spanned by all tensorands of all decomposable tensors appearing in decompositions of $x$ and $f(x)$,
  \item the restriction of $p_i$ to this finite-dimensional subspace has rank one, and
  \item $p_i(x_i)=x_i$.   
\end{itemize}
Since $x$ is fixed by all $p_i\in \fg^M$, and $f$ is a $\fg^M$-module homomorphism, $f(x)$ is preserved by the successive application of $p_1$ up to $p_m$. In other words, $f(x)$ is itself a linear combination of the decomposable tensors $x_{\sigma(1)}\otimes\cdots\otimes x_{\sigma(m)}$. 

This then means that if $X\subset V^*/V_*$ is the $m$-dimensional space spanned by $x_i$, $1\le i\le m$, then $f$ restricts to a monomorphism $X_\lambda\to X_{\lambda'}$ of $\fgl(m)$-modules. The conclusion $\lambda=\lambda'$ follows from classical Schur-Weyl duality for $\fgl(m)$.  
\end{proof}

As a consequence, we get

\begin{corollary}\label{cor.S-W}
  The simple $\fg^M$-modules $(V^*/V_*)_\lambda$ are mutually nonisomorphic as $\lambda$ ranges over all Young diagrams. 
\end{corollary}
\begin{proof}
  Suppose there is an isomorphism $(V^*/V_*)_\lambda\cong (V^*/V_*)_{\lambda'}$ for Young diagrams $\lambda$ and $\lambda'$. The two modules are direct summands of $(V^*/V_*)^{\otimes|\lambda|}$ and $(V^*/V_*)^{\otimes|\lambda'|}$ respectively. Since the identity matrix (which is in $\fg^M$) acts on $(V^*/V_*)^{\otimes m}$ as multiplication by $m$, the Young diagrams $\lambda$ and $\lambda'$ must have an equal number of boxes. Then the isomorphism 
  \begin{equation*}
    (V^*/V_*)_\lambda \cong     (V^*/V_*)_{\lambda'}
  \end{equation*}
can be extended to an endomorphism of $(V^*/V_*)^{\otimes m}$ for $m=|\lambda|=|\lambda'|$ that does not belong to $\bK\left[S_m\right]$, contradicting \Cref{pr.S-W}. 
\end{proof}

%Finally, the target result announced at the beginning of this section follows from all of this: 

\begin{corollary}\label{cor.dist_simples}
  For all ordered triples $(\lambda,\mu,\nu)$, the simple $\fg^M$-modules $ (V^*/V_*)_\lambda\otimes V_{\mu,\nu}$ are mutually nonisomorphic.
\end{corollary}
\begin{proof}
 Let $ (V^*/V_*)_\lambda\otimes V_{\mu,\nu}$ and $ (V^*/V_*)_{\lambda'}\otimes V_{\mu',\nu'}$ be two isomorphic simple modules.  Note that $V^*/V_*$ is a trivial $\fsl\left(V,V_*\right)$-module.  Hence, after restricting to $\fsl\left(V,V_*\right)\subset \fg^M$, the two modules become isomorphic to direct sums of copies of $V_{\mu,\nu}$ and $V_{\mu',\nu'}$ respectively. This implies $(\mu,\nu)=(\mu',\nu')$.

Next, apply \Cref{le.hom_split} to $U=(V^*/V_*)_\lambda$ and $U'=(V^*/V_*)_{\lambda'}$ with $I=\fsl\left(V,V_*\right)$. The hypotheses are easily seen to be satisfied, and the conclusion is that an isomorphism $ (V^*/V_*)_\lambda\otimes V_{\mu,\nu} \cong (V^*/V_*)_{\lambda'}\otimes V_{\mu',\nu'}$ must be of the form $g\otimes \id$ for some isomorphism $g:(V^*/V_*)_{\lambda}\to (V^*/V_*)_{\lambda'}$. The equality $\lambda=\lambda'$ now follows from \Cref{cor.S-W}.
\end{proof}

Finally, we have

\begin{theorem}\label{th.3d_simples}
The simple objects in $\bT^3_{\fgl^M}$ are (up to isomorphism) the tensor products
  \begin{equation*}
    V_{\lambda,\mu,\nu}=(V^*/V_*)_\lambda\otimes V_{\mu,\nu},
  \end{equation*}
and they are mutually nonisomorphic for different choices of ordered triples $(\lambda,\mu,\nu)$. 
%The socle filtration $0=V^{-1}\subseteq V^0\subseteq\ldots$ of $V_{\lambda,\mu,\nu}$ is given by
%\begin{equation*}
%  V^k/V^{k-1} \cong (V^*/V_*)_\lambda\otimes (V^*/V_*)_{\mu^k_{(1)}}\otimes V_{\mu^{|\mu|-k}_{(2)},\nu}.
%\end{equation*}
\end{theorem}
\begin{proof}
  We have already proved that the objects in the statement are mutually nonisomorphic. The fact that they exhaust the simple objects of $\bT^3_{\fgl^M}$ follows from the definition of this category and the fact that every object $(V^*)^{\otimes m}\otimes V^{\otimes n}$ has a finite filtration whose successive quotients are isomorphic to $V_{\lambda,\mu,\nu}$ for various $\lambda$, $\mu$ and $\nu$. 
\end{proof}

%%%%%%%%%%%%%%%%%%%%%%%%%%%%%%%%%%%%%%%%%%%%%%%%%%%%%%%%%%%%%%%%%%%%%%%%%%%%%%%%%%%%%%%%%%%%%%%%%%%%%%%%%%%%%%%%%%
\subsection{Indecomposable injectives}\label{subse.inj_soc}

\Cref{cor.indinjobj} provides a description of the indecomposable injective objects of the category $\bT^3_{\fgl^M}$.   In order to be able to apply it, we first introduce the Grothendieck closure $\overline{\bT}^3_{\fgl^M}$ of $\bT^3_{\fgl^M}$.  By definition, $\overline{\bT}^3_{\fgl^M}$ is the smallest thick, exact Grothendieck subcategory of $\fgl^M\text{-}\mathrm{Mod}$ containing the tensor products of copies of $V$, $V_*$ and $V^*$.  The category $\ol{\bT}^3_{\fgl^M}$ can be made into an ordered Grothendieck category, with $X_i$ and $\Theta_i$ defined as follows.  

The objects $X_i$ are of the form $(V^*/V_*)^{\otimes m}\otimes (V^*)^{\otimes n}\otimes V^{\otimes p}$, where the indexing set $I\ni i$ consists of ordered triples $(m,n,p)$ of nonnegative integers. The morphisms in $\Theta_i$ are the various contractions 
\begin{equation}\label{eq:short1}
(V^*/V_*)^{\otimes m}\otimes (V^*)^{\otimes n}\otimes V^{\otimes p}\to (V^*/V_*)^{\otimes m}\otimes (V^*)^{\otimes (n-1)}\otimes V^{\otimes (p-1)}  
\end{equation}
and all maps
\begin{equation}\label{eq:short2}
(V^*/V_*)^{\otimes m}\otimes (V^*)^{\otimes n}\otimes V^{\otimes p}\to (V^*/V_*)^{\otimes (m+1)}\otimes (V^*)^{\otimes (n-1)}\otimes V^{\otimes p}  
\end{equation}
that send one $V^*$-tensorand onto one $V^*/V_*$-tensorand. Finally, in the partial order on $I$, a triple $(m,n,p)$ is smaller than $(m',n',p')$ if 
\begin{equation}\label{eq:order}
m\ge m',\quad m+n\le m'+n',\quad  n\le n',\quad p\le p',\quad \text{and}\quad m+n-p=m'+n'-p'   
\end{equation}
 (note that the first inequality is reversed). Clearly, this partial order on $I$ satisfies the finiteness condition in \Cref{se.inj} (every element dominates finitely many others). It is routine to check that conditions (1) -- (4) from \Cref{se.inj} hold in this case (\Cref{cor.dist_simples} is needed for condition (3)). Moreover, we have the following observation.

\begin{lemma}\label{le.short}
  The morphisms in the sets $\Theta_i$ described above are short in the sense of \Cref{def.short}.
\end{lemma}
\begin{proof}
For morphisms of the type displayed in \Cref{eq:short1} this amounts to showing that no element of $I$ can be larger than $(m,n-1,p-1)$ and smaller than $(m,n,p)$. For those in \Cref{eq:short2}, on the other hand, we must prove that no element of $I$ is larger than $(m+1,n-1,p)$ and smaller than $(m,n,p)$. We will carry out the first of these two tasks, the other one being entirely analogous. 

Suppose we have
\begin{equation}\label{eq:chain}
  (m,n-1,p-1)\le (a,b,c)\le (m,n,p)
\end{equation}
in our partially ordered set $I$ for some triple $(a,b,c)$ of nonnegative integers. Then, first off, we must have $a=m$. Secondly, $b$ equals either $n$ or $n-1$. In the first case, the condition $a+b-c=m+n-p$ from \Cref{eq:order} forces $c=p$. In the other case, it forces $c=p-1$. Either way, $(a,b,c)$ must be one of the outer ordered triples in \Cref{eq:chain}.  
\end{proof}

\begin{proposition}
	The inclusion functor $\bT^3_{\fgl^M}\subset \ol{\bT}^3_{\fgl^M}$ identifies $\bT^3_{\fgl^M}$ with $\left(\ol{\bT}^3_{\fgl^M}\right)_{\fin}$.
\end{proposition}
\begin{proof}
  This follows immediately from the definition of $\cC_{\fin}$ for an ordered Grothendieck category $\cC$: it is simply the full, thick abelian subcategory on the objects $X_i$.  In the case of $\ol{\bT}_{\fgl^M}^3$, $X_i=(V^*/V_*)^{\otimes m}\otimes (V^*)^{\otimes n}\otimes V^{\otimes p}$ for $i=(m,n,p)$, and the definition of $\left(\ol{\bT}^3_{\fgl^M}\right)_{\fin}$ coincides with the definition of $\bT^3_{\fgl^M}$ given in \Cref{se.prel}. % This, however, is precisely what $\bT^3_{\fgl^M}$ is.  
\end{proof}

Now let $\sym$ be the ring of symmetric functions on countably many variables (see e.g. \cite{sym}, which will be our reference for symmetric functions). The ring $\sym$ is well known to be a Hopf algebra (over $\bZ$), and we will use its distinguished $\bZ$-basis $\{\lambda\}$ of Schur functions labeled by Young diagrams.

As in \cite{Chi14}, denote by $\lambda\mapsto \lambda_{(1)}\otimes\lambda_{(2)}$ the comultiplication on $\sym$. %the Hopf algebra of symmetric functions, identifying Young diagrams with the corresponding Schur functions. 
The tensor $\lambda_{(1)}\otimes\lambda_{(2)}$ is not a decomposable tensor, but we suppress the summation symbol as is customary in the Hopf algebra literature (using so-called Sweedler notation; cf. \cite{Sweedler}). Similarly, we write $\lambda\mapsto \lambda_{(1)}\otimes \lambda_{(2)}\otimes \lambda_{(3)}$ for $(\Delta\otimes\id)\circ\Delta$, etc.

For a generic tensor $\sum_i \mu_i\otimes \nu_i$ in the tensor square $\cat{Sym}^{\otimes 2}$ and nonnegative integers $k$, $\ell$, let $\sum_i\mu^k_i\otimes\nu^\ell_i$ denote the sum of only those summands $\mu_i\otimes\nu_i$ for which $\mu_i$ and $\nu_i$ have $k$ and $\ell$ boxes respectively. 

We write
\begin{equation*}
  0=\soc^0M\subset\soc M=\soc^1M\subset\soc^2M \subset\cdots
\end{equation*}
for the socle filtration of a finite-length object $M$, and denote the semisimple quotient $\soc^i/\soc^{i-1}$ by $\usoc^i$. 

Finally, let $W_{\mu,\nu}$ be the simple object in $\bT_{\End(V)}^2$ corresponding to the Young diagrams $\mu,\nu$. According to \cite[Corollary 5]{Chi14}, the socle filtration of $W_{\mu,\nu}$ over $\fgl^M$ can be written as follows:
\begin{equation*}
  \usoc^k(W_{\mu,\nu}) \cong (V^*/V_*)_{\mu_{(1)}^{k-1}}\otimes V_{\mu_{(2)}^{|\mu|-k+1},\nu}.
\end{equation*}
The proof of \cite[Corollary 5]{Chi14} can be adapted to obtain the socle filtration of $(V^*/V_*)_\lambda\otimes W_{\mu,\nu}$ in $\bT^3_{\fgl^M}$: it is obtained from that of $W_{\mu,\nu}$ by simply tensoring with $(V^*/V_*)_\lambda$. In conclusion, if 
\begin{equation*}
  M=(V^*/V_*)_\lambda\otimes W_{\mu,\nu}
\end{equation*}
then the semisimple subquotient $\usoc^k M$ is isomorphic to
\begin{equation}\label{eq:desc}
  (V^*/V_*)_\lambda\otimes (V^*/V_*)_{\mu_{(1)}^{k-1}}\otimes V_{\mu_{(2)}^{|\mu|-k+1},\nu}.
\end{equation}

We end this section with a description of the indecomposable injective objects in $\bT^3_{\fgl^M}$ and of their socle filtrations. This will be needed below, when we describe the higher ext-groups in $\bT^3_{\fgl^M}$.

\begin{theorem}\label{th.soc_filt}
  The indecomposable injectives in the category $\bT^3_{\fgl^M}$ are 
  \begin{equation*}
    I_{\lambda,\mu,\nu}=(V^*/V_*)_\lambda\otimes (V^*)_{\mu}\otimes V_{\nu},
  \end{equation*}
with respective socles $V_{\lambda,\mu,\nu} = (V^*/V_*)_\lambda\otimes V_{\mu,\nu}$. 

Moreover, the subquotient $\usoc^k(I_{\lambda,\mu,\nu})$ is  isomorphic to
  \begin{equation*}
    \bigoplus_{\ell+r=k-1}\left((V^*/V_*)_\lambda\otimes (V^*/V_*)_{\mu_{(1)}^\ell}\otimes V_{\mu_{(2)}^{|\mu|-k+1},\nu_{(1)}^{|\nu|-r}}\right)^{\oplus \langle \mu_{(3)},\nu_{(2)}\rangle},
  \end{equation*}
where $\langle-,-\rangle:\cat{Sym}\otimes \cat{Sym}\to \bC$ is the pairing making the basis $\{\lambda\}$ orthonormal.  
\end{theorem}

Before going into the proof, let us rephrase the statement so as to make it more explicit. Denoting by $N_{\lambda,\mu}^\nu$ the usual Littlewood-Richardson coefficients (i.e. the multiplicity of $\nu$ in the product $\lambda\mu$ in $\sym$), we have

\begin{thbis}
  The indecomposable injectives in the category $\bT^3_{\fgl^M}$ are (up to isomorphism)
  \begin{equation*}
    I_{\lambda,\mu,\nu}=(V^*/V_*)_\lambda\otimes (V^*)_{\mu}\otimes V_{\nu},
  \end{equation*}
with respective socles $V_{\lambda,\mu,\nu} = (V^*/V_*)_\lambda\otimes V_{\mu,\nu}$. 

    The subquotient $\usoc^k(I_{\lambda,\mu,\nu})$ is  isomorphic to
  \begin{equation}\label{eq.LR}
    \bigoplus_{\ell+r=k-1}\bigoplus_{|\alpha|=\ell}\bigoplus_{|\delta|=r}\bigoplus_\zeta N_{\lambda,\alpha}^\zeta N_{\alpha,\beta}^{\mu}N_{\gamma,\delta}^{\beta}N_{\phi,\delta}^{\nu}~ (V^*/V_*)_{\zeta}\otimes V_{\gamma,\phi},
  \end{equation}
where repeated indices are summed over even if they do not appear under summation signs. 
\end{thbis}
\begin{proof}
Note first that once we prove the claims about the socle filtrations, the fact that the objects $I_{\lambda,\mu,\nu}$ are all (up to isomorphism) indecomposable injectives in $\bT^3_{\fgl^M}$ follows as a direct application of \Cref{pr.inj_env} to $\cC=\ol{\bT}^3_{\fgl^M}$ as explained at the beginning of \Cref{subse.inj_soc}. 

So we are left having to prove the claimed description of the socle filtration of $I_{\lambda,\mu,\nu}$. That the two formulations for it are equivalent is immediate from the definition of the Littlewood-Richardson coefficients in terms of the multiplication on $\sym$ and fact that the pairing $\langle\bullet,\bullet\rangle$ from \Cref{th.soc_filt} is a Hopf pairing, in the sense that 
  \begin{equation*}
    \langle \lambda,\nu_{(1)}\rangle\cdot\langle \mu,\nu_{(2)}\rangle = \langle \lambda\mu,\nu\rangle
  \end{equation*}
for all Young diagrams $\lambda$, $\mu$ and $\nu$. 

The fact that a Jordan-H\"older series for $I_{\lambda,\mu,\nu}$ has the subquotients displayed in \Cref{eq.LR} follows by splicing together the socle filtrations of the objects $(V^*/V_*)_\lambda\otimes W_{\mu,\nu}$ described above (see \Cref{eq:desc}) and the socle filtrations of the indecomposable injectives in $\bT^2_{\End(V)}$ obtained in \cite[Theorem 2.3]{PS}.

Note that for all simple objects in \Cref{eq.LR} (that actually appear
with nonzero multiplicity) we have $|\gamma|=|\mu|-k$. That no such
simple object can appear either later or earlier in the socle
filtration (i.e. in $\usoc^\ell$ for $\ell\ne k$) follows from the fact that 
\begin{equation*}
  \Ext^1((V^*/V_*)_{\zeta}\otimes V_{\gamma,\phi},(V^*/V_*)_{\zeta'}\otimes V_{\gamma',\phi'})
\end{equation*}
cannot be nonzero unless $|\gamma'|-|\gamma|=1$ (see \Cref{th.Kosz} below). 
\end{proof}

%%%%%%%%%%%%%%%%%%%%%%%%%%%%%%%%%%%%%%%%%%%%%%%%%%%%%%%%%%%%%%%%%%%%%%%%%%%%%%%%%%%%%%%%%%%%%%%%%%%%%%%%%%%%%%%%%%
\subsection{Blocks}\label{subse.3d_blocks}

We first make our terminology precise following e.g. \cite[$\S$ 5.1]{ComOst11}:

\begin{definition}\label{def.blocks}
  The \define{blocks} of an abelian category are the classes of the weakest equivalence relation on the set of indecomposable objects generated by requiring that two objects are equivalent if there are nonzero morphisms between them. 
\end{definition}

In this short subsection we classify the blocks of the category $\bT^3_{\fgl^M}$.

\begin{theorem}\label{th.blocks}
  The blocks of $\bT^3_{\fgl^M}$ are indexed by the integers, with the indecomposable injective object $I_{\lambda,\mu,\nu}$ being contained in the block of index $|\lambda|+|\mu|-|\nu|$. 
\end{theorem}
\begin{proof}
  It is clear from \Cref{th.3d_simples,th.soc_filt} that indecomposable injectives $I_{\lambda,\mu,\nu}$ and $I_{\lambda',\mu',\nu'}$ with 
  \begin{equation*}
    |\lambda|+|\mu|-|\nu|\ne |\lambda'|+|\mu'|-|\nu'|
  \end{equation*}
are in different blocks, because they have no isomorphic simple subquotients. 

It remains to show that all $I_{\lambda,\mu,\nu}$ that do share the same indexing integer $|\lambda|+|\mu|-|\nu|$ are in the same block. To see this, note first that $(V^*/V_*)_\lambda$ and $V^*_\lambda$ are in the same block. \Cref{def.blocks} then makes it clear that $I_{\lambda,\mu,\nu}$ is in the same block as some indecomposable direct summand of 
\begin{equation*}
(V^*)^{\otimes (|\lambda|+|\mu|)}\otimes V^{\otimes |\nu|}.  
\end{equation*}
But the classification of blocks in $\bT_{\End_\bK(V)}$ (\cite[Corollary 6.6]{DPS}) implies that \define{all} such direct summands are in the same block as $V_{\mu',\nu'}$ whenever
\begin{equation*}
  |\lambda|+|\mu|-|\nu| = |\mu'|-|\nu'|. 
\end{equation*}
This finishes the proof. 
\end{proof}

%%%%%%%%%%%%%%%%%%%%%%%%%%%%%%%%%%%%%%%%%%%%%%%%%%%%%%%%%%%%%%%%%%%%%%%%%%%%%%%%%%%%%%%%%%%%%%%%%%%%%%%%%%%%%%%%%%

\subsection{Koszulity}
\label{subse.Kosz}

The goal of this section is to prove that the category $\bT^3_{\fgl^M}$ is Koszul. We start with the following

\begin{theorem}\label{th.Kosz}
If $V_{\lambda,\mu,\nu}$ and $V_{\lambda',\mu',\nu'}$ are simple objects of $\bT^3_{\fgl^M}$, then $\Ext^q_{\bT^3_{\fgl^M}}(V_{\lambda,\mu,\nu},
V_{\lambda',\mu',\nu'})\ne 0$ implies $|\mu'|-|\mu| = q$.
\end{theorem}

The proof follows that of \cite[Proposition 5.4]{DPS}, going through an analogue of Lemma 5.3 in loc. cit. In order to state this preliminary result, let us introduce some notation. 

Given three Young diagrams $\lambda$, $\mu$ and $\nu$, let $V_{\lambda,\mu,\nu}^+$ denote the direct sum of the simples of the form $V_{\lambda,\mu',\nu}$ as $\mu'$ ranges over all Young diagrams differing from $\mu$ only in that they have one extra box in one of the rows.

\begin{lemma}\label{le.Kosz}
  For any simple object $V_{\lambda,\mu,\nu}$ in $\bT^3_{\fgl^M}$ there is an exact sequence
  \begin{equation*}
    0\to V_{\lambda,\mu,\nu}^+\to V^*\otimes V_{\lambda,\mu,\nu}\to W\to 0,
  \end{equation*}
where $W$ is a direct sum of simples of the form $V_{\lambda,\mu,\nu'}$ with $|\nu'|=|\nu|-1$ and $V_{\lambda',\mu,\nu}$ with $|\lambda'|=|\lambda|+1$.%\begin{equation*}
%  |\lambda'| - |\lambda| = 0 \text{ or } |\lambda'| - |\lambda| = 1.
%\end{equation*}
\end{lemma}
\begin{proof}
 The $\fg^M$-module $V^*\otimes V_{\lambda,\mu,\nu}$ is an extension of $(V^*/V_*)\otimes V_{\lambda,\mu,\nu}$ by $V_*\otimes V_{\lambda,\mu,\nu}$. By tensoring the exact sequence from \cite[Lemma 5.3]{DPS} with $(V^*/V_*)_{\lambda}$, we see that the $\fg^M$-module $V_*\otimes V_{\lambda,\mu,\nu}$ fits into a short exact sequence 
  \begin{equation*}
        0\to V_{\lambda,\mu,\nu}^+\to V_*\otimes V_{\lambda,\mu,\nu}\to W'\to 0
  \end{equation*}
where $W'$ is a direct sum of simples of the form $V_{\lambda,\mu,\nu'}$ with $|\nu'|=|\nu|-1$. Setting $W:=\left(V_*\otimes V_{\lambda,\mu,\nu}\right)/V_\lambda^+$, we obtain the exact sequence 
\begin{equation}\label{eq.Kosz_aux}
  0\to W'\to W\to (V^*/V_*)\otimes V_{\lambda,\mu,\nu}\to 0.  
\end{equation}
The term $ (V^*/V_*)\otimes V_{\lambda,\mu,\nu}$ is a direct sum of simple modules $V_{\lambda',\mu,\nu'}$ with $|\lambda'|=|\lambda|+1$.  Finally, we can apply \Cref{le.short,le.ext1} to deduce that there are no nontrivial extensions between any two simple constituents of $W'$ and $\left(V^*/V_*\right)\otimes V_{\lambda,\mu,\nu}$, and hence \Cref{eq.Kosz_aux} splits. This concludes the proof. 
\end{proof}

\begin{proof_of_Kosz}
We do induction on $|\mu'|$, the case $|\mu'|=0$ being immediate from the injectivity of the simple modules of the form $V_{\lambda,\emptyset,\nu}$. 

Let $|\mu'|\ge 1$ and assume that the result holds for smaller diagrams $\mu'$. We can find $\beta$ with $|\beta|=|\mu'|-1$ such that $V_{\lambda',\mu',\nu'}$ is a direct summand of $V_{\lambda',\beta,\nu'}^+$. \Cref{le.Kosz} then provides an exact sequence
\begin{equation*}
    0\to V_{\lambda',\beta,\nu'}^+\to V^*\otimes V_{\lambda',\beta,\nu'}\to W\to 0\,,
\end{equation*}
 and the hypothesis $\Ext^q_{\bT^3_{\fgl^M}}(V_{\lambda,\mu,\nu},
V_{\lambda',\mu',\nu'})\ne 0$ leads to one of two possibilities. 

{\bf Case 1: }$\Ext^{q-1}_{\bT^3_{\fgl^M}}(V_{\lambda,\mu,\nu},
W)\ne 0$.  In this case, the equality $|\mu'|-|\mu|=q$ follows from the induction hypothesis and the equality $|\mu'|-|\beta|=1$ since \Cref{le.Kosz} ensures that $|\mu|-|\beta|=q-1$ and $W$ is a direct sum of simples of the form $V_{\kappa,\beta,\delta}$.

{\bf Case 2: }$\Ext^q_{\bT^3_{\fgl^M}}(V_{\lambda,\mu,\nu},
V^*\otimes V_{\lambda',\beta,\nu'})\ne 0$. Let 
\begin{equation*}
  0\to V_{\lambda',\beta,\nu'}\to I_0\to I_1\to\cdots
\end{equation*}
be a minimal injective resolution. Since $V^*\otimes I_*$ is an injective resolution of $V^*\otimes V_{\lambda',\beta,\nu'}$, we must have $\Hom_{\bT^3_{\fgl^M}}(V_{\lambda,\mu,\nu},V^*\otimes I_q)\ne 0$. 

The induction hypothesis ensures that the socle of $I_q$ is a direct sum of simples of the form $V_{\alpha,\beta',\gamma}$ with $|\beta|-|\beta'|=q$, and hence the socle of $V^*\otimes I_q$ is a direct sum of simples of the form $V_{\alpha',\beta'',\gamma'}$ with $|\beta''| = |\beta|-q+1 = |\mu'|-q$.    

This finishes the proof. 
\end{proof_of_Kosz}

What \Cref{th.Kosz} will eventually ensure is that the coalgebra $C$
from \Cref{subse.comod} is Koszul for a certain grading. We elaborate
on this below.

\begin{notation}\label{not.EC}
For a coalgebra $C$ let $EC$ be the $\bN$-graded algebra $\bigoplus_{j}
\Ext^j_{C}(S,T)$ for $S$ and $T$ ranging over all simple comodules of
$C$ with multiplication given by Yoneda composition.   
\end{notation}

$EC$ is analogous to the ext-algebra $\bigoplus_j\Ext^j_A(\bK,\bK)$ for a $\bN$-graded algebra $A=\bK\oplus A^{\geq 1}$, and it similarly controls Koszulity; see e.g. \cite{JLS}, where $EC$ is denoted by $E(C)$. Recall
that one possible characterization of a Koszul algebra is that its ext-algebra be generated in degree one (see e.g. \cite[Definition 1 in $\S$ 2.1]{PP} and the accompanying discussion). Similarly, as a consequence of \Cref{th.Kosz} we get

\begin{theorem}\label{th.Kosz_better}
  For a coalgebra $C$ as in \Cref{th.comod}, the algebra $EC$ is generated in
  degree one.  
\end{theorem}

%Before going into the proof we need one piece of notation. 
%
%
%\begin{notation}
%  For a simple object $S=V_{\lambda,\mu,\nu}$ we denote the Young
%  diagrams $\lambda$, $\mu$ and $\nu$ by $S_1$, $S_2$ and $S_3$
%  respectively. 
%\end{notation}

\begin{proof}
We have to show that for simple objects $S,T\in \bT^3_{\fgl^M}$ the
vector space $\Ext_{\bT^3_{\fgl^M}}^q(S,T)$ is generated by the images of the $q$-fold Yoneda
compositions % of $\Ext_{\bT^3_{\fgl^M}}^1(\bullet,\ast)$ for various simple objects $\bullet,\ast\in\bT^3_{\fgl^M}$. 
\begin{equation*}
	\Ext^1_{\bT^3_{\fgl^M}}\left(T_1,T\right)\otimes \Ext^1_{\bT^3_{\fgl^M}}\left(T_2,T_1\right)\otimes\ldots\otimes \Ext^1_{\bT^3_{\fgl^M}}\left(T_{q-1},T_{q-2}\right) \otimes \Ext^1_{\bT^3_{\fgl^M}}\left(S,T_{q-1}\right) \to \Ext^q_{\bT^3_{\fgl^M}}\left(S,T\right)
\end{equation*}
for simple objects $T_1,\ldots,T_{q-1}\in\bT^3_{\fgl^M}$.

We proceed by induction on $q$, the
case $q=1$ being clear.  By \Cref{th.Kosz} we may as well assume that $T=V_{\lambda',\mu',\nu'}$ and $S=V_{\lambda,\mu,\nu}$ with $|\mu'|-|\mu|=q$. Using the long
exact sequence associated to the short exact sequence
\begin{equation*}
  0\to T\to \widetilde{T}\to \widetilde{T}/T\to 0\,,
\end{equation*}
$\widetilde{T}$ being an injective hull of $T$, we get $\Ext_{\bT^3_{\fgl^M}}^q(S,T)\cong
\Ext_{\bT^3_{\fgl^M}}^{q-1}(S,\widetilde{T}/T)$. Applying \Cref{th.Kosz} again, and
using the fact that the socle $T'$ of $\widetilde{T}/T$ is the direct
sum of all simple constituents $R\cong V_{\alpha,\beta,\gamma}$ of $\widetilde{T}/T$ satisfying
$|\beta|=|\mu'|-1$, we conclude that $\Ext_{\bT^3_{\fgl^M}}^q(S,T)\cong
\Ext_{\bT^3_{\fgl^M}}^{q-1}(S,T')$. Running through how this identification was made,
it follows that $\Ext_{\bT^3_{\fgl^M}}^q(S,T)$ is spanned by the images of the Yoneda compositions 
\begin{equation*}
  \Ext_{\bT^3_{\fgl^M}}^1(R,T)\otimes \Ext_{\bT^3_{\fgl^M}}^{q-1}(S,R)\to \Ext_{\bT^3_{\fgl^M}}^q(S,T)
\end{equation*}
for the various simple summands $R$ of $T'$. This takes care of the
induction step. 
\end{proof}

Dualizing one possible definition of Koszulity for algebras, we give

\begin{definition}\label{def.Kosz_coalg}
An $\bN$-graded coalgebra $C$ is \define{Koszul} if its degree-zero
subcoalgebra is cosemisimple and the $\bN$-graded algebra $EC$ is generated in degree one.   
\end{definition}

\begin{remark}
  There are other definitions of Koszulity for coalgebras in the
  literature (see e.g. \cite[$\S$3.3]{Pos05}) and, just as for algebras, they can be shown to be
  equivalent. 
\end{remark}

The conclusion of the preceding discussion is:

\begin{corollary}\label{cor.Kosz_better}
  The category $\bT^3_{\fgl^M}$ is equivalent to the category of
  finite-dimensional comodules over a Koszul coalgebra. 
\end{corollary}
\begin{proof}
In \Cref{th.comod} we obtained $\bT^3_{\fgl^M}\simeq \cM^C_{\fin}$ by
applying \cite[Theorem 5.1]{Tak77}. That result relies on constructing
an appropriate coalgebra $C$ as the ``coendomorphism coalgebra'' of an
injective cogenerator of the Grothendieck category
$\ol{\bT}^3_{\fgl^M}$ (this is essentially the dual of the
endomorphism ring in the category $\ol{\bT}^3_{\fgl^M}$; we refer to
loc. cit. for details). 

Choosing our injective cogenerator to be the
tensor algebra $\mathrm{T}\big((V^*/V_*)\oplus V^*\oplus V\big)$, we can put an
$\bN$-grading on $C$ by setting $C_n$ to be the space of degree-$n$
coendomorphisms of $\mathrm{T}$. With this grading, $C$ satisfies the
requirements of \Cref{def.Kosz_coalg}. 
\end{proof}

The result analogous to \Cref{cor.Kosz_better} applies to the tensor subcategory $\bT^2_{\fgl^M}$. It is perhaps worth comparing this comodule approach to $\bT^3_{\fgl^M}$ and $\bT^2_{\fgl^M}$ to the description of the latter as a module category in \cite[$\S$ 5]{DPS}. The techniques in loc. cit. can be applied to $\bT^3_{\fgl^M}$ as follows. 

Throughout the rest of this and the next subsections, $\End$, $\Hom$, etc. refer to hom-spaces in $\ol{\bT}^3_{\fgl^M}$.   
We set $\mathrm{T}:=\mathrm{T}\big(\left(V^*/V_*\right)\oplus V^*\oplus V\big)\in \ol{\bT}^3_{\fgl^M}$, with $\mathrm{T}^{\le r}$ denoting its truncation in degrees $\le r$. The associative algebra $\cA^r=\End\left(\mathrm{T}^{\le r}\right)$ can then be realized as a direct summand of $\cA^{r+1}=\End\left(\mathrm{T}^{\le{r+1}}\right)$, and we get a tower of nonunital inclusions of unital algebras. The union $\cA$ of this tower is a nonunital algebra.

\begin{definition}
  A left $\cA$-module is \define{locally unitary} if it is unitary over one of the subalgebras $\cA^r\subset \cA$. 
\end{definition}

We then have the following analogue of \cite[Corollary 5.2]{DPS}, with virtually the same proof.

\begin{theorem}\label{th.mof}
  The functor $\Hom(\bullet,\mathrm{T})$ is an antiequivalence between $\bT^3_{\fgl^M}$ and the category of finite-dimensional locally unitary left $\cA$-modules. 
\qedhere
\end{theorem}

The relationship between \Cref{th.mof} and \Cref{cor.Kosz_better} can be clarified by tracing through Takeuchi's constructions in \cite[$\S$5]{Tak77}. The coalgebra $C$ in \Cref{cor.Kosz_better} is not unique, it is only unique up to Morita equivalence. One possible choice would be the inductive limit
\begin{equation}\label{eq.C_choice}
  \varinjlim_r \Hom\left(\mathrm{T}^{\le r},\mathrm{T}\right)^*,
\end{equation}
in which case the functor implementing the equivalence in \Cref{cor.Kosz_better} would be $\Hom(\bullet,\mathrm{T})^*$. 

In conclusion, $C$ is a subcoalgebra of the graded dual of $\cA$, and the category of $\cA$-modules in \Cref{th.mof} is antiequivalent to $\cM^C_{\cat{fin}}$ via the functor that sends a finite-dimensional vector space to its dual. 

Henceforth, whenever referring to the coalgebra $C$ we specifically mean \Cref{eq.C_choice}.

%%%%%%%%%%%%%%%%%%%%%%%%%%%%%%%%%%%%%%%%%%%%%%%%%%%%%%%%%%%%%%%%%%%%%%%%%%%%%%%%%%%%%%%%%%%%%%%%%%%%%%%%%%%%%%%%%%
\subsection{Universality}
\label{subsec.3d_univ}

Now let $\cD$ be a $\bK$-linear tensor category in the sense of \Cref{subse.tensor} (in particular, symmetric monoidal), with ${\bf 1}_\cD$ denoting its monoidal unit.

Denote by $\bT'^3_{\fgl^M}$ the smallest tensor subcategory of $\bT^3_{\fgl^M}$ that contains $V$, $V^*$, $V^*/V_*$, the surjection $V^*\to V^*/V_*$, the evaluation $V^*\otimes V\to \bK$, and is closed under taking direct summands. Then, \Cref{th.tens_univ} applied to $\cC=\ol{\bT}^3_{\fgl^M}$ allows us to conclude

\begin{proposition}\label{pr.3d_univ}
Let $\cD$ be a tensor abelian category and $F:\bT'^3_{\fgl^M}\rightsquigarrow \cD$ a $\bK$-linear tensor functor. 
\begin{enumerate}
\renewcommand{\labelenumi}{(\alph{enumi})}
\item $F$ extends to a left exact tensor functor $\bT^3_{\fgl^M}\rightsquigarrow\cD$ uniquely up to the obvious notion of natural tensor isomorphism.  
\item If in addition $\cD$ has arbitrary coproducts, $F$ extends uniquely to a left-exact, coproduct-preserving tensor functor $\ol{\bT}^3_{\fgl^M}\rightsquigarrow \cD$. \qedhere
\end{enumerate}
\end{proposition}

The goal of this subsection is to prove the following strengthening of \Cref{pr.3d_univ}.

\begin{theorem}\label{th.3d_univ}
Let $b:x^*\otimes x\to {\bf 1}_\cD$ be a pairing in $\cD$ and $x_*\subseteq x^*$ a subobject.  
\begin{enumerate}
\renewcommand{\labelenumi}{(\alph{enumi})}
\item There is a $\bK$-linear left exact tensor functor $F:\bT^3_{\fgl^M}\rightsquigarrow \cD$, unique up to unique natural tensor isomorphism, sending the pairing $V^*\otimes V\to \bK$ to $b$ and the inclusion $V_*\subset V^*$ to $x_*\subset x^*$.  
\item If in addition $\cD$ has arbitrary coproducts, there is a coproduct-preserving functor $F:\ol{\bT}^3_{\fgl^M}\rightsquigarrow \cD$ as in part (a), unique in the same sense.
\end{enumerate}
\end{theorem}

We will need some preparation. For nonnegative integers $m,n,p$ we denote by $I_{m,n,p}$ the injective object $(V^*/V_*)^{\otimes m}\otimes (V^*)^{\otimes n}\otimes V^{\otimes p}$. If one of the integers $m$, $n$ or $p$ is negative, we set $I_{m,n,p}:=0$.  Recall also the notation $I_{\lambda,\mu,\nu}$ from \Cref{subse.3d_blocks} for the indecomposable injectives in $\bT^3_{\fgl^M}$. 

%As before, $\End$ and $\Hom$ are hom-spaces in $\bT^3_{\fgl^M}$, unless specified otherwise.  

We have the following immediate consequence of the classification of simple and injective objects in $\bT^3_{\fgl^M}$.

%
%
%Before going into it however, note that as an immediate consequence of the universality property from \Cref{th.3d_univ} we obtain the fact that the categories $\bT^3_{\fgl^M}$ associated to various dual-basis pairings $V_*\otimes V\to \bK$ are all equivalent:
%
%
%\begin{corollary}
%For any two diagonalizable nondegenerate pairings $V_*\otimes V\to \bK$ and $W^*\otimes W\to \bK$ the associated categories $\bT^3$ are equivalent as symmetric monoidal categories via functors that identify the two pairings. 
%\qedhere  
%\end{corollary}
%
%

\begin{lemma}\label{le.0}
  For nonnegative integers $m,n,p$ the algebra $\End(I_{m,n,p})$ is isomorphic to the group algebra $\bK[S_m\times S_n\times S_p]$. 
\end{lemma}
\begin{proof}
There is an injective homomorphism $\bK[S_m\times S_n\times S_p]\to \End(I_{m,n,p})$ arising from permuting tensorands in $I_{m,n,p}$.  On the other hand, the dimension of $\End\left(I_{m,n,p}\right)$ is easily seen to equal $\dim \bK\left[S_m\times S_n\times S_p\right]$ by using the isomorphism $\End\left(I_{\lambda,\mu,\nu}\right)=\bK$ for any ordered triple $(\lambda,\mu,\nu)$.  %This means that $i_{m,n,p}$ is an isomorphism in $I_{m,n,p}$.
\end{proof}

\Cref{le.0} can be thought of as describing the degree-zero part of the Koszul coalgebra $C$ from \Cref{cor.Kosz_better} ($C$ is the direct sum over all choices of nonnegative $m,n,p$ of the coalgebras dual to $\bK[S_m\times S_n\times S_p]$). Next, we study those morphisms in $\bT^3_{\fgl^M}$ that contribute to the degree-one component. 

This degree-one component of $C$ clearly contains the contractions 
\begin{equation*}
  \phi_{i,j}:I_{m,n,p}\to I_{m,n-1,p-1}
\end{equation*}
pairing the $i$th $V^*$-tensorand with the $j$th $V$ for $1\le i\le n$ and $1\le j\le p$, and also the surjections
\begin{equation*}
  \pi_{i,j}: I_{m,n,p}\to I_{m+1,n-1,p} 
\end{equation*}
which first cyclically permute the first $i$ $V^*$-tensorands so that the $i$th becomes first, then collapse the first $V^*$ onto $V^*/V_*$, and then permute the last $m-j+1$ $V^*/V_*$-tensorands (for $1\le j\le m+1$) so that the last one becomes $j$th. 

We write $S_{m,n,p}$ for $S_m\times S_n\times S_p$ and $S_{n,p}=S_n\times S_p$ for brevity.

\begin{lemma}\label{le.1}
For nonnegative integers $m,n,r$ we have the following descriptions of hom-spaces in $\bT^3_{\fgl^M}$. 
\begin{enumerate}
\renewcommand{\labelenumi}{(\alph{enumi})}
\item $\Hom(I_{m,n,p},I_{m,n-1,p-1})$ is isomorphic to $\bK[S_{m,n,p}]$ as a bimodule over 
  \begin{equation*}
    \End(I_{m,n-1,p-1})\cong \bK[S_{m,n-1,p-1}] \text{ and } \End(I_{m,n,p})\cong\bK[S_{m,n,p}],   
  \end{equation*}
where we regard $\phi_{i,j}$ as the generator of the right $S_{m,n,p}$-module structure for some $1\le i\le n$ and $1\le j\le p$ and $S_{n-1,p-1}$ is the subgroup of $S_{n,p}$ fixing $i$ and $j$. 
\item $\Hom(I_{m,n,p},I_{m+1,n-1,p})$ is isomorphic to the induced representation 
  \begin{equation*}
    \bK[S_{m+1,n,p}]\cong \mathrm{Ind}_{S_m}^{S_{m+1}}\bK[S_{m,n,p}]=\bK[S_{m+1}]\otimes_{\bK[S_m]}\bK[S_{m,n,p}]
  \end{equation*}
as a bimodule over
  \begin{equation*}
    \End(I_{m+1,n-1,p})\cong \bK[S_{m+1,n-1,p}] \text{ and } \End(I_{m,n,p})\cong\bK[S_{m,n,p}],   
  \end{equation*}
where we regard $\pi_{i,j}$ as a generator of $\bK[S_{m,n,p}]$ as a the right $S_{m,n,p}$-module, and $S_m\subset S_{m+1}$ and $S_{n-1}\subset S_n$ are the subgroups fixing $j$ and $i$ respectively.  
\end{enumerate} 
\end{lemma}
\begin{proof}

{\bf (a)} Fixing $i$ and $j$ as in the statement, there is a morphism
\begin{equation*}
  \alpha:\bK[S_{m,n,p}] \to \Hom(I_{m,n,p},I_{m,n-1,p-1})
\end{equation*}
of $(\bK[S_{m,n-1,p-1}],\bK[S_{m,n,p}])$-bimodules that sends $1\in \bK[S_{m,n,p}]$ to $\phi_{i,j}$. 
The surjectivity of $\alpha$ follows from \Cref{le.CX_full}.

We are now left having to prove that $\alpha$ is injective, or in other words that the compositions $\phi_{i,j}\circ \sigma$ are linearly independent for $\sigma\in S_{m,n,p}$. 
Suppose some linear combination 
\begin{equation}\label{eq.le.11}
  \sum_{\sigma\in S_{m,n,p}} a_\sigma (\phi_{i,j}\circ \sigma)
\end{equation}
vanishes. Then split $V$ as a direct sum $V_1\oplus\cdots\oplus V_q$ of nontrivial subspaces for $q\ge n+p$ and fix $i'$ and $j'$, $1\leq i'\leq n$, $1\leq j'\leq p$. %and $j'$ in the same indexing sets as $i$ and $j$ respectively. 
To see that all $a_\sigma$ must vanish, we apply \Cref{eq.le.11} to decomposable tensors whose $(i')$th $V^*$-tensorand is $v^*$ and whose $(j')$th $V$-tensorand $v$ belongs to $V_1$ and satisfies $v^*(v)=1$ while all other $V^*$ and $V$-tensorands lie respectively in $V^*_k$, $V_\ell$ for distinct $k$ and $\ell$. Indeed, this shows that $a_\sigma$ vanishes when $\sigma(i')=i$ and $\sigma(j')=j$.  To complete the proof we apply the same argument to all possible choices of $i'$ and $j'$.

{\bf (b)} The proof is very similar to that of part (a). Once more, we can define a morphism
\begin{equation*}
  \beta:\bK[S_{m+1}]\otimes_{\bK[S_m]}\bK[S_{m,n,p}]\to \Hom(I_{m,n,p},I_{m+1,n-1,p})
\end{equation*}
of $(\bK[S_{m+1,n-1,p}],\bK[S_{m,n,p}])$-bimodules sending $1\in \bK[S_{m,n,p}]\subset \bK[S_{m+1}]\otimes_{\bK[S_m]}\bK[S_{m,n,p}]$ to $\pi_{i,j}$ for fixed $i$ and $j$ in the appropriate indexing sets.

We conclude that $\beta$ is onto just as in part (a). As far as injectivity is concerned, again an argument analogous to that from (a) will do. 
\end{proof}

\Cref{le.1} is analogous to \cite[Lemma 6.1 + part (1) of Lemma 6.2]{DPS}. The following result, in turn, is our version of \cite[part (2) of Lemma 6.2]{DPS}; it follows immediately from \Cref{le.1}.

\begin{lemma}\label{le.11}
  The tensor products of hom-spaces from \Cref{le.1} can be described as follows. 
  \begin{enumerate}
     \renewcommand{\labelenumi}{(\alph{enumi})}
      \item The space  
        \begin{equation}\label{eq.11a}
          \Hom(I_{m,n-1,r-1},I_{m,n-2,p-2})\otimes_{\End(I_{m,n-1,p-1})}\Hom(I_{m,n,p},I_{m,n-1,p-1})
        \end{equation}
        is isomorphic to $\bK[S_{m,n,p}]$ as an $(S_{m,n-2,p-2},S_{m,n,p})$-bimodule. 
      \item The space  
        \begin{equation}\label{eq.11b}
          \Hom(I_{m,n-1,p-1},I_{m+1,n-2,p-1})\otimes_{\End(I_{m,n-1,p-1})}\Hom(I_{m,n,p},I_{m,n-1,p-1})
        \end{equation}
        is isomorphic to $\bK[S_{m+1,n,p}]\cong \mathrm{Ind}_{S_m}^{S_{m+1}}\bK[S_{m,n,p}]$ as an $(S_{m+1,n-2,p-1},S_{m,n,p})$-bimodule. 
      \item The space  
        \begin{equation}\label{eq.11c}
          \Hom(I_{m+1,n-1,p},I_{m+1,n-2,p-1})\otimes_{\End(I_{m+1,n-1,p})}\Hom(I_{m,n,p},I_{m+1,n-1,p})
        \end{equation}
        is isomorphic to $\bK[S_{m+1,n,p}]\cong \mathrm{Ind}_{S_m}^{S_{m+1}}\bK[S_{m,n,p}]$ as an $(S_{m+1,n-2,p-1},S_{m,n,p})$-bimodule.  
      \item The space  
        \begin{equation}\label{eq.11d}
          \Hom(I_{m+1,n-1,p},I_{m+2,n-2,p})\otimes_{\End(I_{m+1,n-1,p})}\Hom(I_{m,n,p},I_{m+1,n-1,p})
        \end{equation}
        is isomorphic to $\bK[S_{m+2,n,p}]\cong \mathrm{Ind}_{S_m}^{S_{m+2}}\bK[S_{m,n,p}]$ as an $(S_{m+2,n-2,p},S_{m,n,p})$-bimodule. \qedhere 
   \end{enumerate}
\end{lemma}

We now analyze the quadratic part of the graded coalgebra $C$ in \Cref{cor.Kosz_better}. Dually, this means describing the kernels of the composition maps defined on the spaces in \Cref{le.11}.

\begin{lemma}\label{le.2}
  The composition maps defined on the spaces from \Cref{le.11} can be described as follows. 
  \begin{enumerate}
     \renewcommand{\labelenumi}{(\alph{enumi})}
      \item The map    
        \begin{equation*}
          \Cref{eq.11a}\to \Hom(I_{m,n,p},I_{m,n-2,p-2})
        \end{equation*}
        is onto, and its kernel is generated by
        \begin{equation*}
          \phi_{n-1,p-1}\otimes \phi_{n,p} - (\phi_{n-1,p-1}\otimes \phi_{n,p})\circ(n,n-1)(p,p-1)
        \end{equation*}
        as an $(S_{m,n-2,p-2},S_{m,n,p})$-bimodule, where $(n,n-1) \in S_n\subset S_{m,n,p}$ and $(p,p-1) \in S_p\subset S_{m,n,p}$. 
      \item The map  
        \begin{equation*}
          \Cref{eq.11b}\oplus \Cref{eq.11c}\to \Hom(I_{m,n,p},I_{m+1,n-2,p-1})
        \end{equation*}
        is onto, and its kernel is generated by
        \begin{equation*}
          \pi_{m+1,n-1}\otimes\phi_{n,p}-(\phi_{n-1,p}\otimes\pi_{m+1,n})\circ(n,n-1)
        \end{equation*}
        as an $(S_{m+1,n-2,p-1},S_{m,n,p})$-bimodule, where $(n,n-1)\in S_n\subset S_{m,n,p}$.  
      \item The map  
        \begin{equation}\label{eq.2c_pre}
          \Cref{eq.11d}\to \Hom(I_{m,n,p},I_{m+2,n-2,p})
        \end{equation}
        is onto, and its kernel is generated by
        \begin{equation}\label{eq.2c}
          (m+2,m+1)\circ(\pi_{m+2,n-1}\otimes \pi_{m+1,n}) - (\pi_{m+2,n-1}\otimes\pi_{m+1,n})\circ (n,n-1)
        \end{equation}
        as an $(S_{m+2,n-2,p},S_{m,n,p})$-bimodule, where $(m+2,m+1) \in S_{m+2}\subset S_{m+2,n-2,p}$ and $(n,n-1) \in S_n\subset S_{m,n,p}$. 
   \end{enumerate}
\end{lemma}
\begin{proof}
  The surjectivity of the three maps follows, just as in the proof of \Cref{le.1}, from the tensor analogue of \Cref{le.CX_full}. The other claims are very similar in nature, so proving part (c) will be sufficient for illustration purposes. 

Identify the two tensorands of \Cref{eq.11d} with $\bK[S_{m+2,n-1,p}]$ and $\bK[S_{m+1,n,p}]$ as in part (b) of \Cref{le.1} by using as generators the maps $\pi_{m+2,n-1}$ and $\pi_{m+1,n}$ respectively. This then identifies the space \Cref{eq.11d} with $\bK[S_{m+2,n,p}]$ as in part (d) of \Cref{le.11}. 

With this identification in hand, the element \Cref{eq.2c} is just
\begin{equation*}
  (m+2,m+1)-(n,n-1)\in (\bK[S_{m+2}]\otimes \bK\otimes \bK)+(\bK\otimes \bK[S_n]\otimes \bK)\subseteq \bK[S_{m+2,n,p}].
\end{equation*}
The $(S_{m+2,n-2,p},S_{m,n,p})$-bimodule generated by it coincides with the right $S_{m+2,n,p}$-module it generates; its dimension equals $\frac12\dim\bK\left[S_{m+2,n,p}\right]=\frac12(m+2)!n!p!$. We are trying to show that this module equals the kernel of the onto composition map \Cref{eq.2c_pre}. It is certainly contained in this kernel, so it is enough to show that 
\begin{equation*}
\dim\Hom(I_{m,n,p},I_{m+2,n-2,p}) \geq \frac 12 (m+2)!n!p!\,.
\end{equation*}

This can be seen as follows. Consider a typical element
\begin{equation}\label{eq.2_elt}
  (x_1\cdots x_m)\otimes (y_1\cdots y_n)\otimes (z_1\cdots z_p)\in (V^*/V_*)^{\otimes m}\otimes (V^*)^{\otimes n}\otimes V^{\otimes p},
\end{equation}
where we have suppressed some tensor product symbols for ease of reading. Assume that the $z_i$s are linearly independent, and similarly the $x_i$s are jointly linearly independent with the images of the $y_i$s through $V^*\to V^*/V_*$. 

An element of $S_{m+2,n,p}$ will now permute the $z_i$s, permute the $y_i$s, permute the $x_i$s, and then insert the classes in $V^*/V_*$ of the first two $y_i$s (after having permuted the $y_i$s) among the $x_i$s at two spots. It is easy to see from our linear independence condition on the vectors that this leads to linearly independent vectors for different permutations and different points of insertion. There are $m!n!p!$ choices for the permutations, and $\binom{m+2}{2}=\frac{(m+2)(m+1)}{2}$ choices for the points of insertion. This gives the desired bound. 
\end{proof}

\begin{proof_of_3d_univ}
  In view of \Cref{pr.3d_univ}, it suffices to show that the data of a pairing $x^*\otimes x\to {\bf 1}_\cD$ and a subobject $x_*\subseteq x^*$ can be extended to a tensor functor $F:\bT'^3_{\fgl^M}\to \cD$, unique up to symmetric monoidal natural isomorphism, such that
  \begin{equation*}
    F(V^*\otimes V\to \bK)\cong x^*\otimes x\to{\bf 1}_\cD\quad\text{and}\quad F(V_*\subset V^*)\cong x_*\subseteq x^*.
  \end{equation*}
The uniqueness is a consequence of the fact that $\bT'^3_{\fgl^M}$ is generated as a tensor category by $V^*/V_*$, $V^*$ and $V_*$. As for the existence, there is an obvious candidate for the extension that we would like to show is well defined. In order to do this we need to check that all of the relations between morphisms in $\bT'^3_{\fgl^M}$ are satisfied in $\cD$. 

By Koszulity (\Cref{cor.Kosz_better}), it is enough to check this for quadratic relations. The conclusion now follows from \Cref{le.2} which shows that the only relations are those that hold in any tensor category. 
\end{proof_of_3d_univ}

%
%
%
%%%%%%%%%%%%%%%%%%%%%%%%%%%%%%%%%%%%%%%%%%%%%%%%%%%%%%%%%%%%%%%%%%%%%%%%%%%%%%%%%%%%%%%%%%%%%%%%%%%%%%%%%%%%%%%%%%%
%\subsection{More general pairings}
%\label{subsec.cardin}
%
%
%Everything we have done so far can be carried out in the more general context of pairings $V_*\otimes V\to \bK$ that afford dual bases (of arbitrary cardinality) for $V$ and $V_*$. Minor adaptations are required here and there (e.g. in the proof of \Cref{pr.S-W} we need to use $I\times I$ matrices for possibly uncountable sets $I$ rather than $I=\bN$), but these are routine and we leave them to the reader. 
%
%In particular, we get a universality statement analogous to \Cref{th.3d_univ} in this more general setup. Since that universality implies uniqueness up to equivalence, we have
%
%
%\begin{theorem}\label{th.cardin}
%Let $V$ be an arbitrary vector space, and $V_*\otimes V\to \bK$ a pairing that induces dual bases on $V$ and $V_*$. Then, the category $\bT^3_{\fgl^M(V,V_*)}$ defined as above is symmetric monoidally equivalent to the one obtained in the countable-dimensional case. 
%  \qedhere
%\end{theorem}
%
%
%
%

%%%%%%%%%%%%%%%%%%%%%%%%%%%%%%%%%%%%%%%%%%%%%%%%%%%%%%%%%%%%%%%%%%%%%%%%%%%%%%%%%%%%%%%%%%%%%%%%%%%%%%%%%%%%%%%%%%
\subsection{Split subcategories and derived full faithfulness}
\label{subsec.3d_split-subc-deriv}

We provide here an application of the universality property from \Cref{th.3d_univ}.

Consider the full tensor subcategory $\bT^2_{\fgl^M}\subset \bT^3_{\fgl^M}$. Since the embedding functor $\iota:\bT^2_{\fgl^M}\rightsquigarrow \bT^3_{\fgl^M}$ is exact, it induces a functor $D\iota:D^+\left(\bT^2_{\fgl^M}\right)\rightsquigarrow D^+\left(\bT^3_{\fgl^M}\right)$ by applying $\iota$ straightforwardly to complexes.% It turns out that $\iota$ has the following strong full faithfulness property.   

\begin{theorem}\label{th.3d_derived_ff}
The derived functor $D\iota$ is fully faithful.   
\end{theorem}

\begin{remark}\label{re.3d_faith}
  In other words, the maps $\Ext^i_{\bT^2_{\fgl^M}}(x,y)\to \Ext^i_{\bT^3_{\fgl^M}}(\iota x,\iota y)$ induced by $\iota$ are isomorphisms. 
\end{remark}

We postpone the proof briefly in order to make a few remarks. 

By \Cref{th.3d_univ} the inclusion functor $\iota$ is split by a unique left exact tensor functor $R:\bT^3_{\fgl^M}\rightsquigarrow \bT^2_{\fgl^M}$ sending the pairing $V^*\otimes V\to \bK$ to the pairing $\mathrm{p}:V_*\otimes V\to \bK$ and turning the exact sequence 
\begin{equation*}
  0\to V_*\to V^*\to V^*/V_*\to 0
\end{equation*}
into 
\begin{equation*}
  0\to V_*\to V_*\to 0\to 0.
\end{equation*}
Here, `split' means that $R\circ\iota$ is naturally isomorphic as a tensor functor to the identity. In fact, we have 

\begin{lemma}\label{le.R_right_adj}
  The functor $R$ is right adjoint to $\iota$. 
\end{lemma}
\begin{proof}
The universality in \Cref{th.3d_univ} ensures the existence of isomorphisms of the tensor functors $\eta:\id_{\bT}\to R\circ\iota$ and $\varepsilon:\iota\circ R\to\id_{\bT^3_{\fgl^M}}$. While $\eta$ is the isomorphism referred to above, $\varepsilon_V:V\to V$ is the identity and $\varepsilon_{V^*}:V_*\to V^*$ is the inclusion $V_*\subset V^*$.   

The same universality result then shows that both

\begin{equation*}
  \begin{tikzpicture}[auto,baseline=(current  bounding  box.center)]
    \path[anchor=base] (0,0) node (1) {$\iota$} +(3,0) node (2) {$\iota\circ R\circ\iota$} +(6,0) node (3) {$\iota$};
         \draw[->] (1) to node[auto,pos=.5] {$\id\circ \eta$} (2);
         \draw[->] (2) to node[auto,pos=.5] {$\varepsilon\circ\id $} (3);
  \end{tikzpicture}
\end{equation*}
and
\begin{equation*}
  \begin{tikzpicture}[auto,baseline=(current  bounding  box.center)]
    \path[anchor=base] (0,0) node (1) {$R$} +(3,0) node (2) {$R\circ \iota\circ R$} +(6,0) node (3) {$R$};
         \draw[->] (1) to node[auto,pos=.5] {$\eta\circ \id$} (2);
         \draw[->] (2) to node[auto,pos=.5] {$\id\circ \varepsilon$} (3);
  \end{tikzpicture}
\end{equation*}
are identities. This exhibits $\eta$ and $\varepsilon$ as the unit and respectively counit of an adjunction, as claimed. 
\end{proof}

\begin{proof_of_3d_derived_ff}
  We now know by \Cref{le.R_right_adj} that $\iota$ is left adjoint to $R$. Since $\iota$ is also exact, $D\iota$ and $DR:D^+(\bT^3_{\fgl^M})\rightsquigarrow D^+\left(\bT^2_{\fgl^M}\right)$ similarly constitute an adjunction ($DR$ denoting the right derived functor to $R$). 

The fact that $\eta:\id_{\bT^2_{\fgl^M}}\to R\circ\iota$ is an isomorphism ensures that the same is true for the unit $D\eta:\id_{D^+\left(\bT^2_{\fgl^M}\right)}\to DR\circ D\iota$. In other words, the left adjoint $D\iota$ is fully faithful.    
\end{proof_of_3d_derived_ff}

Incidentally, $\bT^2_{\fgl^M}$ sits inside $\bT^3_{\fgl^M}$ as a split subcategory in another way: By \Cref{th.3d_univ} again we have a left exact symmetric monoidal functor $j:\bT^2_{\fgl^M}\rightsquigarrow \bT^3_{\fgl^M}$ sending $V_*\otimes V\to \bK$ to $V^*\otimes V\to \bK$. Similarly, there is a functor $S:\bT^3_{\fgl^M}\rightsquigarrow\bT^2_{\fgl^M}$ with $S\circ j\cong \id$ and $S\left(V_*\right)=0$ (unlike $R$ above which annihilates the quotient $V^*/V_*$).  In a similar vein to \Cref{le.R_right_adj} we have

\begin{lemma}
$S$ is right adjoint to $j$. 
\qedhere  
\end{lemma}

Moreover, the argument in the proof of \Cref{th.3d_derived_ff} can be adapted to prove

\begin{theorem}\label{th.3d_derived_ff_bis}
The derived functor $Dj:D^+\left(\bT^2_{\fgl^M}\right)\rightsquigarrow D^+\left(\bT^3_{\fgl^M}\right)$ is fully faithful.  
\qedhere  
\end{theorem}

\begin{remark}
Since $\bT^3_{\fgl^M}$ is roughly speaking obtained from its subcategory $j\left(\bT^2_{\fgl^M}\right)$ by adding a subobject $V_*\subset V^*\in j\left(\bT^2_{\fgl^M}\right)$, the embedding $j:\bT^2_{\fgl^M}\rightsquigarrow \bT^3_{\fgl^M}$ is superficially similar to the restriction functor from the category of representations of a linear algebraic group to that of a parabolic subgroup.  For a parabolic subgroup $P\subset G$ of a reductive algebraic group there is a similar derived full faithfulness phenomenon, whereby restriction from $\cat{rep}(G)$ to $\cat{rep}(P)$ induces isomorphisms between all $\Ext^i$ (see e.g. \cite[Corollary II.4.7]{Jan03}). The significance of \Cref{th.3d_derived_ff_bis} is that it confirms this familiar phenomenon in our setting, thus strengthening the analogy. 
\end{remark}

%%%%%%%%%%%%%%%%%%%%%%%%%%%%%%%%%%%%%%%%%%%%%%%%%%%%%%%%%%%%%%%%%%%%%%%%%%%%%%%%%%%%%%%%%%%%%%%%%%%%%%%%%%%%%%%%%%

\subsection{Koszul self-duality}

One of the striking results discovered in \cite{DPS} about the Koszul category $\bT^2_{\fgl^M}$ generated by $V$ and $V_*$ is that it is \define{self-dual}, in the sense that the graded algebra that is the $\bT^2_{\fgl^M}$-analogue of $\cA$ from \Cref{th.mof} above is isomorphic to the opposite of its quadratic dual. This result is then used in \cite{DPS} in the computation of the exts between simple objects of $\bT^2_{\fgl^M}$.

We proceed in a similar way. Our first assertion is that we similarly have self-duality in the context of $\bT^3_{\fgl^M}$.

\begin{theorem}\label{th.self-dual}
  The Koszul coalgebra $C$ defined by \Cref{eq.C_choice} is isomorphic to the opposite of its Koszul dual.  
\end{theorem}
\begin{proof}
  It suffices to prove the self-duality statement for the truncations $\cA^r$ of the algebra $\cA$ in the discussion at the end of \Cref{subse.Kosz}. This will come  as a consequence of \Cref{le.0,le.1,le.11,le.2} applied to $(m,n,p)$ with $m+n+p=r$. The proof is similar to that of \cite[Lemma 6.4]{DPS}.

Following the discussion on quadratic duals in \cite[$\S$2.8]{BGS}, denote by $\bK$ the semisimple degree-zero component of $\cA^r$ and by $U$ the degree-one component. Moreover, $\cA=\mathrm{T}(U)/(R)$, where
%\begin{equation*}
%  \mathrm{T}^\bullet V = \bK\oplus V\oplus (V\otimes V)\oplus\cdots
%\end{equation*}
%is the tensor algebra of $V$ over $\bK$ and 
$R$ is the space of quadratic relations, i.e. the direct sum of the spaces described in \Cref{le.2} for $m+n+p=r$. 

The quadratic dual $(\cA^r)^!$ is then $\mathrm{T}(U^*)/(R^\perp)$, where %$V^*=\Hom_k(V,k)$ and 
$R^\perp\subset U^*\otimes U^*$ is the annihilator of $R\subset U\otimes U$. First, identify the direct summands $\bK[S_{m',n',p'}]\subset U$ from \Cref{le.1} with their duals by making the elements of $S_{m',n',p'}$ self-dual. It is now easy to see that the map 
\begin{equation}\label{eq.sign}
  d:S_{m',n',p'}=S_{m'}\times S_{n'}\times S_{p'} \to U\cong U^*,\,
 \sigma=(\sigma_1,\sigma_2,\sigma_3) \mapsto \mathrm{sgn}(\sigma_2)\sigma
\end{equation}
will intertwine the multiplication on $\cA^r$ and the opposite multiplication on $(\cA^r)^!$ provided the map $d$ respects the relations. We once more argue this last point just in case (c) of \Cref{le.2}, as the other two cases are analogous. 

Following the proof of \Cref{le.2}, identify $\Hom(I_{m+1,n-1,p},I_{m+2,n-2,p})$ and $\Hom(I_{m,n,p},I_{m+1,n-1,p})$ with $\bK[S_{m+2,n-1,p}]$ and $\bK[S_{m+1,n,p}]$ as in part (b) of \Cref{le.1}, using $\pi_{m+2,n-1}$ and $\pi_{m+1,n}$ as generators respectively. The generator \Cref{eq.2c} of the relevant summand of $R$ is then identified with $(m+2,m+1)-(n,n-1)$, whose annihilator in $U\cong U^*$ is generated by $(m+2,m+1)+(n,n-1)$. This concludes the argument. 
\end{proof}

As a consequence of this, we now have the dimensions of the higher extension groups between simple objects of $\bT^3_{\fgl^M}$. The following result is analogous to \cite[Corollary 6.5]{DPS} and \cite[Proposition 3.3.3]{SS}.

\begin{corollary}\label{cor.self-dual}
  For any two simple objects in $\bT^3_{\fgl^M}$ and any $q\ge 0$ we have
  \begin{equation}\label{eq.self-dual}
    \dim\ \Ext^q(V_{\lambda,\mu,\nu},V_{\lambda',\mu',\nu'}) = \text{ multiplicity of }V_{\lambda,\mu^\perp,\nu}\text{ in }\usoc^{q+1}(I_{\lambda',(\mu')^\perp,\nu'}),
  \end{equation}
where $\perp$ denotes the transpose of a Young diagram. 
\end{corollary}
\begin{proof}
The left-hand side of \Cref{eq.self-dual} is a summand of the degree-$q$ component $EC_q$ of the algebra $EC$ from \Cref{not.EC}. On the other hand, \cite[Theorem 2.10.1]{BGS} claims that $EC$ is the opposite of the Koszul dual of $\cA$, which in turn is equal to $\cA$ by \Cref{th.self-dual}. 

The proof can now proceed as in \cite[Corollary 6.5]{DPS}, using the fact that the right hand side of \Cref{eq.self-dual} is the dimension of $p_{\lambda,\mu^\perp,\nu}\cA p_{\lambda',(\mu')^\perp,\nu'}$ where $p_{\bullet,\bullet,\bullet}$ are the idempotents in $\cA$ corresponding to the respective indecomposable injectives in $\bT^3_{\fgl^M}$, and the observation that the identification $(\cA^!)^{\mathrm{opp}}\cong EC$ from the proof of \Cref{th.self-dual} matches the idempotent of $EC$ corresponding to $V_{\lambda,\mu,\nu}$ to that of $\cA$ corresponding to $V_{\lambda,\mu^\perp,\nu}$ because of the sign in \Cref{eq.sign}. 
\end{proof}

As a consequence of \Cref{th.soc_filt}, we can rephrase the result as follows.

\begin{corbis}
The dimension of $\Ext^q(V_{\lambda,\mu,\nu},V_{\lambda',\mu',\nu'})$
equals
  \begin{equation*}
        \bigoplus_{\ell+r=q}\bigoplus_{|\alpha|=\ell}\bigoplus_{|\delta|=r}
        N_{\lambda',\alpha}^\lambda
        N_{\alpha,\beta}^{(\mu')^\perp}N_{{\mu^\perp},\delta}^{\beta}N_{\nu,\delta}^{\nu'},
  \end{equation*}
where repeated indices are summed over even if they do not appear under summation signs. 
\qedhere  
\end{corbis}

We illustrate the contents of this subsection with an example.

\begin{example}
  In \Cref{cor.self-dual}, let $\lambda=(2)={\tiny\begin{ytableau}
 \phantom{a} &\phantom{a}
\end{ytableau}}\,$, $\mu=\emptyset$, and $\nu=(1)={\tiny\begin{ytableau}
 \phantom{a} 
\end{ytableau}}\,$, and also
  \begin{equation*}
    \lambda'=\mu'=\nu'=(1). 
  \end{equation*}
Then, since the injective hull $I_{\lambda',(\mu')^\perp,\nu'}$ equals $(V^*/V_*)\otimes V^*\otimes V$, \Cref{cor.self-dual} says that $$\Ext^q\left(V_{(2),\emptyset,(1)},V_{(1),(1),(1)}\right)=0$$ for all $q\neq1$, and that $\dim\Ext^1\left(V_{(2),\emptyset,(1)},V_{(1),(1),(1)}\right)=1$. In other words, a nonzero vector in $\Ext^1\left(V_{(2),\emptyset,(1)},V_{(1),(1),(1)}\right)$ arises from the exact sequence
\begin{equation*}
  0\to V_{(1),(1),(1)}\to I_{(1),(1),(1)}\to V_{(2),\emptyset,(1)}\oplus V_{(1,1),\emptyset,(1)}\oplus V_{(1),\emptyset,\emptyset}\to 0
\end{equation*}
where $(1,1)={\tiny\begin{ytableau}
 \phantom{a} \\
 \phantom{a}
\end{ytableau}}\,$, 
and there are no higher extensions between $V_{(2),\emptyset,(1)} = S^2(V^*/V_*)\otimes V$ and $V_{(1),(1),(1)}=(V^*/V_*)\otimes(\mathrm{ker}\,\mathrm{p})$ ($\mathrm{p}:V_*\otimes V\to \mathbb{K}$ is our original pairing).  

Similarly, the nonzero vectors in $\Ext^1\left(V_{(1,1),\emptyset,(1)},V_{(1),(1),(1)}\right)$ and $\Ext^1\left(V_{(1),\emptyset,\emptyset},V_{(1),(1),(1)}\right)$ arise from the same exact sequence, and $\Ext^q\left(V_{(1,1),\emptyset,(1)},V_{(1),(1),(1)}\right)=\Ext^q\left(V_{(1),\emptyset,\emptyset},V_{(1),(1),(1)}\right)=0$ for $q\neq 1$.
\end{example}

%%%%%%%%%%%%%%%%%%%%%%%%%%%%%%%%%%%%%%%%%%%%%%%%%%%%%%%%%%%%%%%%%%%%%%%%%%%%%%%%%%%%%%%%%%%%%%%%%%%%%%%%%%%%%%%%%%
%%%%%%%%%%%%%%%%%%%%%%%%%%%%%%%%%%%%%%%%%%%%%%%%%%%%%%%%%%%%%%%%%%%%%%%%%%%%%%%%%%%%%%%%%%%%%%%%%%%%%%%%%%%%%%%%%%
\section{The four-diagram category}\label{sec.four}

It is interesting to study the full, thick, tensor category $\bT^4_{\fgl^M}$ of $\fg^M$-modules obtained by adjoining $\ol{V}=(V_*)^*$ to $\bT^3_{\fgl^M}$. This category is not an ordered tensor category in the sense of \Cref{subse.tensor}. One way to see this is to show that the indecomposable injective objects in the Grothendieck closure of $\bT^4_{\fgl^M}$ have infinite length.  We will possibly provide the proof of the latter fact elsewhere.

In this paper, we restrict ourselves to showing that $\bT^4_{\fgl^M}$ is a finite-length category and to describing its simple objects. We begin by proving the following

\begin{proposition}\label{pr.simples}
  For any two Young diagrams $\lambda$ and $\mu$, the $\fg^M$-module $(V^*/V_*)_\lambda\otimes\left(\ol{V}/V\right)_\mu$ is simple.  
\end{proposition}

As before, we think of $\fg^M$ as consisting of $\bZ_{>0}\times \bZ_{>0}$ matrices with finitely many nonzero entries in each row and column, and of $V^*/V_*$ as the space of infinite row vectors modulo those with finitely many nonzero entries, the action of $\fg^M$ on $V^*/V_*$ being given by the formula $gv=-vg$. In the same spirit, $\ol{V}/V$ consists of $\bZ_{>0}$-indexed column vectors that are regarded as equivalent up to changing finitely many entries; the action of $\fg^M$ on $\ol{V}/V$ is usual left multiplication.

We need the following auxiliary result.

\begin{lemma}\label{le.simples}
  For any choice of $n$-tuples $v_i,v_i'\in V^*/V_*$ and $w_i, w_i'\in \ol{V}/V$, $1\le i\le n$, if the $v_i$ and $w_i$ are linearly independent then there is an element $g\in \fg^M$ such that
  \begin{equation*}
    gv_i = v_i',\quad gw_i=w_i',\ \forall i.
  \end{equation*}
\end{lemma}
\begin{proof}
The fact that our matrices are infinite allows enough freedom to construct such a $g$ by recursion. Throughout the proof we will conflate elements of $V^*/V_*$ and representatives in $V^*$ (and similarly for $\ol{V}/V$ and $\ol{V}$ respectively). 

First, fix the first column of $g$ so as to ensure that, for every $i$,
the first entry of $gv_i$ equals the first entry of $v_i'$; this is possible by the linear independence assumption. Now the first entry of the first row is fixed, but we can still fill in the first row so as to ensure that the first entries of $gw_i$ equal those of $w_i'$ respectively. 

Now repeat the procedure, completing the second column appropriately
and then the second row, etc. At every step we are ensuring by
construction that our rows and columns have finitely many non-zero
entries, hence $g\in \fg^M$ as desired.   
\end{proof}

As in \cite[$\S$ 2]{Chi14}, for any subset $I\subseteq \bZ_{>0}$ we say that an element of $V^*/V_*$ or $\ol{V}/V$ is \define{concentrated in $I$} or \define{$I$-concentrated} if the indices of all but finitely many of its nonzero entries are in $I$. 

We also say that an element of some tensor power $(V^*/V_*)^{\otimes n}$ is $I$-concentrated if it is a sum of decomposable tensors whose individual tensorands are $I$-concentrated. This also applies to images of Schur functors like $(V^*/V_*)_\lambda$, which we think of as subspaces of a tensor power. 

Finally, a matrix in $\fg^M$ is $I$-concentrated if all of its nonzero entries are indexed by $I\times I$.

\begin{proof_of_simples}
  Denote $M=V^*/V_*$, $N=\ol{V}/V$ and let $x\in M_\lambda\otimes N_\mu$ be a nonzero element. 

Following the same strategy as in the proof of \cite[Proposition 1]{Chi14}, we can assume that $x$ can be written as a sum of decomposable tensors $x_i\otimes y_i$ with $x_i\in M_\lambda$ and $y_i\in N_\mu$ concentrated in two disjoint infinite subsets $I$ and $J$ respectively of $\bZ_{>0}$. 

The Lie subalgebra of $\fg^M$ consisting of block-diagonal matrices with blocks concentrated in $I$ and $J$ is isomorphic to $\fg^M\oplus\fg^M$, and the two summands act separately on the $I$- and $J$-concentrated parts of $M_\lambda$ and $N_\mu$ respectively. The simplicity of these latter modules over $\fg^M$ (\cite[Proposition 1]{Chi14}) implies that the submodule of $M_\lambda\otimes N_\mu$ generated by $x$ contains all elements of the form $c_\lambda m\otimes c_\mu n$ where $c_\lambda$ and $c_\mu$ are the Young symmetrizers corresponding to $\lambda$ and $\mu$ respectively, and $m\in M^{\otimes |\lambda|}$ and $n\in N^{\otimes |\mu|}$ range over the $I$- and $J$-concentrated elements respectively. 

If $m\otimes n\in M^{\otimes|\lambda|}\otimes N^{\otimes|\mu|}$ is such that the individual tensorands of $m$ are linearly independent, we can apply \Cref{le.simples} repeatedly to $m\otimes n$ to get any other $m'\otimes n$. Similarly, we can change the $n$ tensorand at will. 

Applying this to elements $c_\lambda m\otimes c_\mu n$ as above, which we know are contained in the submodule of $M_\lambda\otimes N_\mu$, shows that indeed \define{all} of 
\begin{equation*}
  M_\lambda\otimes N_\mu=c_\lambda M^{\otimes|\lambda|}\otimes c_\mu N^{\otimes|\mu|}
\end{equation*}
is contained in this submodule. 
\end{proof_of_simples}

As a consequence of \cite[Lemma 3]{Chi14} this implies

\begin{corollary}\label{cor.4diag_aux}
  For any four Young diagrams $\lambda$, $\mu$, $\nu$, $\eta$, the $\fgl^M$-module
  \begin{equation*}
    (V^*/V_*)_\lambda\otimes\left(\ol{V}/V\right)_\mu\otimes V_{\nu,\eta} \in \bT^4_{\fgl^M}
  \end{equation*}
is simple.
\qedhere
\end{corollary}

Finally, as announced above, we show that we have found all of the simple objects and that $\bT^4_{\fgl^M}$ is a finite-length category.

\begin{proposition}\label{pr.4diag}
  All objects in $\bT^4_{\fgl^M}$ have finite length and the objects from \Cref{cor.4diag_aux} exhaust its simple objects. 
\end{proposition}
\begin{proof}
  This is immediate from \Cref{cor.4diag_aux} and the observation that all tensor products of copies of $V$, $V_*$, $V^*$ and $\ol{V}$ admit finite filtrations whose quotients are of the form described in \Cref{cor.4diag_aux}.  
\end{proof}

%%%%%%%%%%%%%%%%%%%%%%%%%%%%%%%%%%%%%%%%%%%%%%%%%%%%%%%%%%%%%%%%%%%%%%%%%%%%%%%%%%%%%%%%%%%%%%%%%%%%%%%%%%%%%%%%%%
%%%%%%%%%%%%%%%%%%%%%%%%%%%%%%%%%%%%%%%%%%%%%%%%%%%%%%%%%%%%%%%%%%%%%%%%%%%%%%%%%%%%%%%%%%%%%%%%%%%%%%%%%%%%%%%%%%

%\bibliography{Mackey}{}
\bibliographystyle{plain}
\addcontentsline{toc}{section}{References}

\def\cftil#1{\ifmmode\setbox7\hbox{$\accent"5E#1$}\else
  \setbox7\hbox{\accent"5E#1}\penalty 10000\relax\fi\raise 1\ht7
  \hbox{\lower1.15ex\hbox to 1\wd7{\hss\accent"7E\hss}}\penalty 10000
  \hskip-1\wd7\penalty 10000\box7}

\Addresses

\end{document}